\documentclass[12pt]{article}
\usepackage[T2A]{fontenc}
\usepackage[utf8]{inputenc}
\usepackage{amsmath}
\usepackage{cite}
\usepackage{amsmath,}
\usepackage{amssymb,amsfonts}
\usepackage{algorithmic}
\usepackage{textcomp}
\usepackage{amsthm}

\newtheorem{thm}{Theorem}
\newtheorem{cor}[thm]{Corollary}
\newtheorem{lem}[thm]{Lemma}
\newtheorem{prop}{Proposition}
\newtheorem{exm}[thm]{Example}
\newtheorem{rem}[thm]{Remark}

\DeclareMathOperator{\tr}{\mathrm{tr}}
\DeclareMathOperator{\F}{\mathrm{\mathbb{F}}}

\begin{document}
\begin{center}
	{\LARGE\LARGE \bf
Extended Special Linear group $ESL_2(\mathbb{F})$ and matrix equations in $SL_2(\mathbb{F})$, $SL_2(\mathbb{Z})$, $GL_2(\mathbb{F}_p)$,
 and $M_2(\mathbb{F}_p)$
}

\medskip

	{\bf Skuratovskii Ruslan \textsuperscript{[0000-0002-5692-6123]}}
\smallskip

	{\it ruslcomp@gmail.com }

	


\end{center}	
\bigskip

\section{Abstract }
First time, we introduce Extended special linear group $ESL_2(\mathbb{F})$, which is generalization of the matrix group $SL_2(\mathbb{F})$, where $\mathbb{F}$ is arbitrary perfect field.
We show that $ESL_2(\mathbb{F})$ is a set of all square matrix roots from $SL_2(\mathbb{F})$.
We generalize the group of unimodular matrices \cite{Un} and find its structure.

We generalize the group of unimodular matrices and find a structure of extended symplectic group $ESp_2(\mathbb{R})$ as well as generilized group of unimodular matrices.

The criterions of roots existing for different classes of matrix --- simple and semisimple matrixes from $ SL_2({\mathbb{F}})$, $ SL_2({\mathbb{Z}})$  and $ GL_2({\mathbb{F}})$ are established.
So our criterions oriented on general class of matrix depending of the form of minimal and characteristic polynomials, moreover proposed criterion  holds in $GL_2(\mathbb{F})$ where $\mathbb{F}$ is an arbitrary field.

The problems of square root from group element existing in  $SL_2(\mathbb{F}_p)$, $SL_2(\mathbb{F}_p)$ and $GL_2(\mathbb{F}_p)$ for arbitrary prime $p$ are solved in this paper. The similar goal of root finding was reached in the GM algorithm adjoining an $n$-th root of a generator \cite{GM}
results in a discrete group
 for group $SL(2,R)$, but we consider this question over finite field $\F_p$. Well known the Cayley-Hamilton method \cite{Pell} for computing the square roots of the matrix $M^n$ can give answer of square roots existing over a finite field only after computation of $det M^n$ and some real Pell-Lucas numbers by using Bine formula. Over method gives answer about existing $\sqrt{ M^n}$ without exponenting $M$ to $n$-th power. We only use the trace of $M$ or only eigenvalues of $M$. We have expanded the well-known Cayley-Hamilton method to provide a complete description of the roots in all cases \cite{DON}.

 The authors of \cite{Amit} considered criterion to be square only for the case $\F_p$ is a field of characteristics not equal 2. We solve this problem even for fields $\mathbb{F}_2$ and $\mathbb{F}_{2^n}$.
  The criterion to $g \in SL_2 (\mathbb{F}_2)$ be square in $SL_2(\mathbb{F}_2)$ was not found by them what was declared in a separate sentence in \cite{Amit}.
  In case of field with characteristic 0 there is only the Anisotropic case of  group  $SL_1(\mathbb{Q})$,  where $\mathbb{Q}$ is a quaternion division algebra over $k$  was considered in \cite{Amit}.
Also in \cite{Amit} the split case of $SL_2(k)$ and its powers was considered, where under group splitting authors mean Bruhat decomposition is the double coset decomposition of the group $SL_2(k)$ with respect to the subgroup Borel $B$, consisting of upper triangualar matrix from $SL_2(k)$.
Moreover authors of \cite{Amit}  find recursive formula of matrix root in $SL_2(k)$ only for case of field $char(k) \neq 2$, but our formula solves this problem for a  $\mathbb{F}_p$ with arbitrary characteristic. Futhermore, our formula is analytical and does not require sequential recursive calculation.

The analytical formula of square roots of 2-nd, 3-rd and 4-th power in $SL_2(\mathbb{F}_p)$, $SL_2({\mathbb{F}_p})$ are found. Moreover
we managed to find the recursive formula for calculating the root of an arbitrary degree $n$ from an arbitrary square matrix of dimension 2.
We  find solutions for all cases of roots existing for the root formula from \cite{DON} the equation $X^2=A$, in contrast to the formula proposed in \cite{DON} for which the authors did not find roots in the singular case, when 0 appears in the denominator of the formula.


We investigate a condition of a matrix quadraticity depending on its Jordan structure and spectrum.


\textbf{Key words}: extended special linear group,  equation in matrix group,  splittable  extension, formula of square roots in linear groups, extended symplectic group, set of squares in matrix group, criterion of  square root existing in $SL_2(\mathbb{F}_p)$.\\
\textbf{2000 AMS subject classifications}:  20B27, 20E08, 20B22, 20B35, 20F65, 20B07.

\section{Introduction }

Firstly we introduce new algebraic group that is $ESL_2(\mathbb{F}_p)$ which contains all solutions of  $X^2= A$ for  $A \in  SL_2(\mathbb{F}_p)$.
Then we research the conditions of matrix equation solvability $X^2= A$ in $SL_2(\mathbb{F}_p), GL_2(\mathbb{F}_p)$ and one of splitting extension of $SL_2(\mathbb{F}_p)$ that is $ESL_2(\mathbb{F}_p)$  \cite{SkSq, Square}. 

Our statements  can be easy reformulated for these groups over the field $\mathbb{R}$ so it lead us to arguments of solving of discreteness problem \cite{GM, GM1} in some subgroups of $SL(2,\mathbb{R})$.

One method of computing square roots of two-by-two matrices was presented in
\cite{DON} but under unsolved by him condition which $\sqrt{A}$  exists.
Moreover, the author did not find solutions in the limiting case when the denominator $\tr A  \pm 2 \sqrt{detA}$ tends to 0 so we describe this cases. Also formulas for roots of  3-rd and forth powers were established by us.

But we find not only all such conditions but the method of square root computation for $n \times n$ matrix. Moreover we consider this question also in groups over finite fields. Also we indicate in which group $\sqrt{A}$  lies.


We consider a more general case then \cite{SkSq} consisting in the whole group $G= SL_2(\mathbb{F}_q)$ because of we do not  provide additional condition of  splitting.  Also the authors considered separetely conjugacy classes in $SL_2(\mathbb{F}_q )$  \cite{Amit}  such as:  central classes, split regular semisimple classes, non-semisimple classes, anisotropic regular semisimple classes. For each case the criterion of solvability of equation is provided. In the last two cases Bruhat decomposition is applied.

The previous investigations \cite{Nort, DON} claims that for some matrices in $S{{L}_{2}}\left( \F_{2} \right)$ have not square root in $S{{L}_{2}}\left( {\F_{2}} \right)$. Now we make group classification of roots distribution in which root could exist in splittable extension of group $S{{L}_{2}}\left( {{\text{F}}_{p}} \right)$ over the same field viz it is in $ES{{L}_{2}}\left( {{\text{F}}_{p}} \right)$.
We investigate root distribution of $A \in S{{L}_{2}}\left( {{\text{F}}_{p}} \right)$ by cosets of $ES{{L}_{2}}\left( {{\text{F}}_{p}} \right)$ by the normal subgroup $S{{L}_{2}}\left( {{\text{F}}_{p}} \right)$.

The action of subgroup of new group $ESL_2(\mathbb{F}_p)$ introduced here also arose without description of group structure and generators in the topology.
Namely, if $G$ is a Morse-Bott foliation on the solid Klein bottle K into 2-dimensional Klein bottles parallel to the boundary and one singular circle $S^1$ then such group appears as leaf preserving diffeomorphisms for foliations $G$  \cite{MaksB}.

In many geometrical groups there are automorphisms preserve hyperbolic distance (hyperbolic metric) and hyperbolic angles, furthermore they may change orientation of space as well as keep it permanent \cite{Steinberg}.

In hyperbolic geometry there are groups preserve hyperbolic length \cite{Khan} and orientation as well as changes orientation, in particular projective special linear group $PS{{L}_{2}}(\mathbb{R})$ and $S{{L}_{2}}(\mathbb{R})$ possessing changing orientatio due to action of $S{{L}_{2}}(\mathbb{R})$ is non-faithful because of $PSL_2(\mathbb{R})$ is a homomorphic image of $SL_2(\mathbb{R})$ with non-trivial kernel. A proposed by us group $ES{{L}_{2}}(\mathbb{R})$ also  preserves hyperbolic length \cite{Khan} 




One of interesting algorithmic problem of combinatorial group theory was solved by Roman'kov \cite{Rom}.
 It was problem of determining for any element $g\in G$ is $g$ a commutator for free nilpotent group $N_r$ of arbitrary rank $r$ with class of nilpotency 2 \cite{Rom}.
The analogous problem can be formulated for  $SL_n(\mathbb{F}_q)$, $GL_n(\mathbb{F}_q)$ and $ESL_n(\mathbb{F}_q)$ over a set of squares.


The problem of the solvability of an equation over a group is well known \cite{Klyach, Kunyav}. We consider the same problem with additional constrains on the solvability of an equation of the form $X^2=A$  in a group.

Question of root existing in different forms appears in the Purtzitsky-Rosenberger trace minimizing algorithm \cite{Purtz, GM} it was considered roots and rational powers of one or both generators of    in non-elementary two generator
discrete subgroups of $PSL_2(\mathbb{R})$ found by the GM algorithm. But we solve existing root problem for arbitrary element of $SL_2(\mathbb{F}_p)$.

Also such criterion of root existing for $SL_2(\mathbb{F}_p)$, $SL_2(\mathbb{R})$ and $GL_2(\mathbb{F}_p)$ are established.
This criterion is a stricter version of the formulated question for group extensions how large must an overgroup of a given group be in order to contain a square root of any element of the initial group $G$, which was considered in the paper of Anton A. Klyachko and D. V. Baranov \cite{Klyach}. Our criterion gives the answer that such extension is $ESL_2(\mathbb{F})$ for $SL_2(\mathbb{F})$.

 The main result of this paper about criterion of quadraticity can be extended to larger dimension matrices having a Jordan structure constructed of blocks of dimension 2 or 1. Also, our result for a semisimple matrix of dimension 2 can be generalized to a semisimple matrix of higher dimension.


In this research we continue our previous investigation \cite{SkAbst, SkAbst21,  SkMahn, Square, SkTula}.
%





%

\section{Preliminaries }
To show the importance of studying this group we notice some topological manifolds in which $ESL(\mathbb{R})$ subgroups appear.

An action of a subgroup of $ ESL_2(\mathbb{R})$ appears in leaf preserving diffeomorphism group which is called
foliated leaf preserving in Morse-Bott foliation on the solid torus \cite{MaksTor} of simplest Morse-Bott foliations.
But this action was defined geometrically by symmetries with respect to meridian and parallel of torus and infinite shift on torus also corresponding 3 matrix to these elements were given.
 Indeed Morse-Bott foliation on solid torus \cite{MaksTor} $T = S^{-1} \times D^2$ into 2-tori parallel to the boundary and one singular circle consists of elements presented by matrices with determinant 1 and -1 by author who characterize it as a subgroup of the whole $GL_2(R)$ group. But now we characterize it more precisely as a subgroup of smaller group $ESL_2(R)<GL_2(R)$. The diffeomorphisms group of this manifold  posses the subgroup  $\mathcal{G}$ described in the geometrical terms,  where the actions of shifts, symmetries relative to a parallel to a meridian appear, where shift is generated by $\left(\begin{array}{rr}
   1\,\,&0 \\
  1 \,\,&1
\end{array} \right)$ which is called by reflection.
Symmetries relative to a parallel and a meridian are defined by matrices $\left(\begin{array}{rr}
   1\, &0 \\
  0 \, &-1
\end{array} \right), \ \
\left(\begin{array}{rr}
   -1\,\,&0 \\
  0 \,\,&1
\end{array} \right)$ correspondingly.
This matrices generate group $\mathcal{G}$ which is a proper subgroup of $ESL_2(\mathbb{Z})$.

Define the algebraic properties and structures of $ESL_2(\mathbb{F}_p)$ in the next item.

\textbf{\bf Definition 1.}
{\sl The set of matrices
\begin{equation}\label{ESL}
 \left\{ {{M}_{i}}:Det({{M}_{i}})=\pm 1,  {M}_{i} \in GL_2(\mathbb{F}_p) \right\} \end{equation}

   forms \textbf{extended special linear group}  in $GL_2(\mathbb{F}_p)$ and is denoted by  ${ESL}_2(\mathbb{F}_p)$.

As it is studied by us, $\texttt{ESL}_2(\mathbb{F}_p) \cong SL_2(\mathbb{F}_p)\rtimes \,{{\mathbb{C}}_{2}}$, where $\,{{\mathbb{C}}_{2}}$ is generated by reflection $\left(
\begin{array}{cc}
-1 & 0 \\
 \,\,0 & 1 \\
\end{array}
\right)$. The involution from the top-subgroup $\,{{\mathbb{C}}_{2}} \simeq   \left\langle \left(
\begin{array}{cc}
-1\,\, & 0 \\
0 & 1 \\
\end{array}
\right) \right\rangle  $ induces the sign of automorphism in $Aut\left( SL_2(\mathbb{F}_p) \right)$.
}

It is obvious that ${ESL}_2(\mathbb{F}_p)$ possess presentation in $GL_2(\mathbb{F}_p)$ by matrices described in Definition \ref{ESL} to show it
we establish the homomorphism $\psi$ from $SL_2(\mathbb{F}_p)\rtimes \,{{\mathcal{C}}_{2}}$ to $\texttt{ESL}_2(\mathbb{F}_p)$. We construct $\psi$ sending elements of the semidirect product containing matrix $i$ as an element of top group ${{\mathcal{C}}_{2}}$ in quotient class of ${}^{ {\texttt{ESL}_2(\mathbb{F}_p)}} \diagup {}_{ \texttt{SL}_2(\mathbb{F}_p)}$ having determinant $-1$ and an with matrix $E$ in the qoutient class having determinant 1.

Matrices  with determinant -1 correspond to the elements changing Euclidean space orientation.
As it was found in our study of the roots in matrix groups, solutions of $X^2=A$ arise in defined above group $ESL_2(\mathbb{F}_p)$, where $A \in SL_2(\mathbb{F}_p)$.
 We can spread the definition of $ESL_2(\mathbb{F}_p)$ on case of
  matrices over the arbitrary field  $\mathbb{F}$ as well as over the ring $\mathbb{Z}$.

 Justification of $ SL_2(\mathbb{F}_p),  SL_2(\mathbb{Z})$ extensions existence is based on  the description $Aut\left( SL_2((\mathbb{F_p}) \right), Aut\left( SL_2(\mathbb{Z}) \right) $ and its subgroups of order 2.
 In similar way we can extend $ SL_n(\mathbb{F})$ to $ ESL_n (\mathbb{F}_p)$.

$SL_2(\mathbb{F}_p)$ is subgroup of index 2 in $ESL_2(\mathbb{F}_p)$ so its normality is established.

The existence of a non-trivial homomorphism $\varphi :\,\,{{\mathbb{Z}}_{2}}\to Aut\left( S{{L}_{2}}(\mathbb{Z}) \right)$, as well as $\phi :\,\,{{\mathbb{Z}}_{2}}\to Aut\left( S{{L}_{2}}({{\mathbb{F}}_{p}}) \right)$ can be proved by indicating an element of order 2 in the automorphisms of base group that is the kernel of the semidirect product we want to construct.
There is countergradient automorphism in $S{{L}_{2}}\left( \mathbb{Z} \right)$ that is $\varphi :\,M\to {{\left( {{M}^{T}} \right)}^{-1}}$ or alternating automorphism of order 2 acting by conjugating $\varphi :\,M\to D^{-1}MD$, where $D=\left( \begin{array}{rr}
   1\,\,\, &0 \\
  0\,\, & -1 \\
\end{array} \right)$ and is called by diagonal automorphism \cite{Mersl}.

Recall the \textbf{definition} of $\mathbf{TI-subgroup}$ \cite{Suds,Zu}.  Let $G$ be a group and $A < G$, then $A$ is called $\mathbf{TI-}$subgroup iff  $A \cap A^g = e$ for each $g \in G \setminus N_G(A)$.

\begin{rem}
Subgroup $\mathbb{C}_{2}$ is $\mathbf{TI-subgroup}$ and antinormal subgroup.
\end{rem}

\begin{proof}
In view of ${\mathbb{C}}_{2}$ is one generated then its centralizer coincides with its normalizer. One easy can verify that centralizer consists of all diagonal matrices  from $ESL_2(\mathbb{F}_p)$.
Let us find a structure of such normalizer $N_{ESL_2(\mathbb{F}_p)} ({\mathbb{C}}_{2})$.
In view of e.v. is  invariant under conjugation by non-singular matrix over field the normalizer of top subgroup ${\mathbb{C}}_{2}$ in $ESL_2(\mathbb{F}_p)$ consists of  all diagonal matrices  from $ESL_2(\mathbb{F}_p)$ and permutational matrix ${{\mathcal{P}}}=\left( \begin{array}{rr}
   0\,\,\, &1 \\
  1\,\, & 0 \\
\end{array} \right)$.
We assume that $N_{ESL_2(\mathbb{F}_p)} ({\mathbb{C}}_{2}) \simeq   D(SL_2(\mathbb{F}_p))\rtimes \,{\mathcal{P}}$, where $D(SL_2(\mathbb{F}_p))$ diagonal subgroup of $ESL_2(\mathbb{F}_p)$.

For the rest of elements condition of $A \cap A^g = e$ for each $g \in ESL_2(\mathbb{F}_p) \setminus N_{ESL_2(\mathbb{F}_p)} ({\mathcal{C}}_{2})$ holds. Thus, $\mathbb{C}_{2}$ is $\mathbf{TI-subgroup}$, hence $\mathbb{C}_{2}$ is antinormal subgroup.
\end{proof}

It is obvious that there is a homomorphism in matrix presentation of $ESL_2(\mathbb{F}_p)$ from the semidirect product defining the extension of the group $ SL_2(\mathbb{F}_p)$  as the kernel of the semidirect product, by a group of two matrices, one $E$ the second reflection matrix $i$ inducing changes in the sign of the determinant in $ESL_2(\mathbb{F}_p)$.

 $SL_2(\mathbb{Z})$ is a normal subgroup of $ESL_2(\mathbb{Z})$, as being the kernel of the determinant, which is a group homomorphism whose image is the multiplicative group $ \{-1, +1 \}$.

\begin{rem} It is obvious that orthogonal group $O_2(k) < ESL_2(k)$, where $k$ is a field but $O_2(k) \ntriangleleft  ESL_2(k)$ \cite{Linear, LinearShpringer}.
\end{rem}
In fact, the action by conjugation of elements from the $ESL_2(k)$ does not preserve angles and does not fixe non-degenerate quadratic and Hermitian forms.

We briefly introduce the minimal set of generators and new relations in   $ESL_2\left( \mathbb{Z} \right)$ \cite{AutSLp} i.e. this group over  integer ring.
We denote a matrix of shift
$\left(
\begin{array}{cc}
1 & 1 \\
0 & 1 \\
\end{array}
\right)$
  by $s$ and
  $\left(
\begin{array}{cc}
0 & -1 \\
1 &  0 \\
\end{array}
\right)$
    as $t$  they generate $SL_2\left( \mathbb{Z} \right)$, new generator
$\left(
\begin{array}{cc}
-1 & 0 \\
 0 & 1 \\
\end{array}
\right)$
is denoted by $i$. Each relation of $SL_2\left(  \mathbb{Z} \right)$ holds.
Then new relation is
${{i}}si^{-1}=s^{-1}$. The second relation is ${{i}}ti^{-1}=t^{-1}$ and the rest of them are $t^4=i^2=e$. The order of $s$ is $\infty$ because $s$ is a shift. Note, that $\mathbb{C}_{2}= <i>$.

Note that elements $i$ and $t$ are orthogonal because of $ti=0$. Some interesting relation in this terms of the kernel subgroup $SL(2,\mathbb{Z})$ are ${{t}^{2}}=-E$,  ${{t}^{-2}}s{{t}^{2}}=s$.

Existence justification of such extension of $SL_2(\mathbb{Z})$ by $\mathbb{C}_2$  is based on $Aut(SL_2(\mathbb{Z}))$ \cite{Mart, Mersl, Gur} structure which is splitting  extension $SL_2(\mathbb{Z})$  by $\mathbb{Z}$.
As well known the group of outer automorphisms of $S{{L}_{n}}(\mathbb{Z})$ is semidirect products of the form  $S{{L}_{n}}(\mathbb{Z}){{\rtimes }_{\varphi }}\mathbb{Z}$  and its isomorphism type depends only on  $[\varphi ]\in Out(S{{L}_{n}}(\mathbb{Z}))$.  Since $Aut(SL_2(\mathbb{Z}))$ contains an element of order 2 that is $t^2$  therefore homomorphism from top group that is cyclic group $\mathbb{C}_2 = <i>$ of order 2  in $Aut(SL_2(\mathbb{Z}))$ exists.

The action by right multiplication on  $\left(
\begin{array}{cc}
-1 & 0 \\
0 & 1 \\
\end{array}
\right)$ of a matrix from  $S{{L}_{n}}(\mathbb{Z})$ inducing automorphism inverting sing of first column of matrix $A$.  This  automorphism invert sign of $det(A)$.


 A new geometrical group $\mathcal{G}$ appears as subgroup in the group ${{D}^{lp}}\left( F \right)$  of diffeomorphisms group of $T$ and  $\left[ 0;\text{ }1 \right]$ on ${{C}^{\infty }}(T,\left[ 0;\text{ }1 \right])$ and now be characterized by us in more structural and exact way. Because of the authors \cite{MaksTor} consider $\mathcal{G}$ as subgroup of very wide group $GL(2, \mathbb{Z})$ consisting of matrices for which the vector $(0, 1)$ is eigen with eigenvalue $\pm 1$, which was defined as:
$$\mathcal{G}=\,\left\{ \,\left( \begin{array}{rr}
   \varepsilon \,\,&0 \\
  m\,\, & \delta
\end{array} \right)\left| m\in \mathbb{Z},\,\,\,\varepsilon ,\delta \in \left\{ \pm 1 \right\} \right. \right\}.$$


But $\mathcal{G}$ is a proper subgroup of $ESL_2(\mathbb{Z)}$ that is more special then whole $G{{L}_{2}}\left( \mathbb{Z} \right)$,  moreover $ES{{L}_{2}}\left( \mathbb{Z} \right)$ has as a kernel of semidirect product a proper subgroup of $S{{L}_{2}}\left( \mathbb{Z} \right)$, and $\mathcal{G}$ has in role of kernel a proper subgroup of $S{{L}_{2}}\left( \mathbb{Z} \right)$, because of  $det\left( \mathcal{G} \right)=\pm 1$.
Furthermore the concept of new group  ${ESL}_{2}\left( \mathbb{Z} \right)$ admits us to obtain a structural characterization and set of generators with relations for $\mathcal{G}$.
We take in consideration first generator of $\mathcal{G}$ that is involutions generating symmetry of torus  with respect to the parallel. It is represented by matrix
 $t=\left(\begin{array}{rr}
   1\,\,&0 \\
  0 \,\,&-1
\end{array} \right)$ and generators of the top subgroup of ${ ESL}_{2}\left( \mathbb{Z} \right)$ which is denoted by $i=\left(\begin{array}{rr}
  -1\,\,&0 \\
  0 \,\,\,&1
\end{array} \right)$. One easy can verify that third generator $D$ of $\mathcal{G}$ can be derived from generators of $ESL_2(\mathbb{Z)}$ in the following way $t = - E \times i $, because $-E \in ESL_2(\mathbb{Z)}$.

Now using concept of new group $ESL_2(\mathbb{Z)}$ allows us to give exact and structural characterization of group $\mathcal{G}$ which contains in ${{D}^{lp}}\left( F \right)$. For this goal we consider subgroup of ${ESL}_2(\mathbb{Z)}$ with kernel $K\simeq \left\langle \left( \begin{array}{rr}
   1\,\,\,&0 \\
  1\,\,\,\,&1 \\
\end{array} \right) \right\rangle$. Since $K\simeq \left\langle \left( \begin{array}{rr}
   \,1\,\,\, & 0 \\
  \,1\,\,\, & 1
\end{array} \right) \right\rangle \simeq \mathbb{Z}$  then $AutK\simeq {{\mathbb{Z}}_{2}}$ and therefore homomorphism from subgroup $\left\langle i \right\rangle $ as well as from  subgroup $\left\langle t \right\rangle $ to $AutK$ exist. One easy can check that
$i
\left(
\begin{array}{rr}
1\ & 0 \\
1\ & 1
\end{array}
\right)
i^{-1}=
\left(
\begin{array}{rr}
1\ & 0 \\
-1\ & 1
\end{array}
\right)
=
\left(
\begin{array}{rr}
1\ & 0 \\
1\ & 1
\end{array}
\right)
^{-1}
$
and rest of conjugations remain $K$ invariant.
Thus, we find a structure of $\mathcal{G}$ which, up to a way to define a semidirect product, is $\mathcal{G}\,\simeq K\ltimes \left\langle t,i \right\rangle $.
 An important fact that $K\ltimes \left\langle t,i \right\rangle $
  is a subgroup  in ${ESL}_{2}(\mathbb{Z})$. Top subgroup of $\mathcal{G}$ has 2 generators but kernel subgroup $K$ is one generated, unlike the  kernel in $(\mathbb{Z})$ having 2 generators. If we denote $\left( \begin{array}{rr}
   \,1\,\,\,& 0 \\
     \,1\,\,\,& 1
\end{array} \right)$  then the relations are following $isi={{s}^{-1}},\,\,tst={{s}^{-1}},\,\,\,{{t}^{2}}={{s}^{2}}=e$.

We denote by e.v. --- \textbf{eigenvalue}s.
Let $\mu_A$ be minimal polynomial of $A$.

A polynomial $P(X)$ over a given field $K$ is separable if its roots are distinct in an algebraic closure of $K$, that is, the number of distinct roots is equal to the degree of the polynomial.
\emph{Simple matrix} is a matrix such that characterstic polynomial is separable.

Recall that matrix $A$ is
called \textbf{semisimple} if $\mu_A$ is a product of distinct monic irreducible and
separable polynomials; if moreover all these irreducible polynomials
have degree 1, then $A$  is called split semisimple or diagonalizable \cite{ LinearShpringer, Linear}. 

We denote iff --- necessary and sufficient condition,
e,v.  --- eigenvalue.

\subsection {Some possible applications in topology }

Geometrical transformations corresponding to matrices that form the subgroup of the introduced here $S{{L}_{2}}\left( \mathbb{R} \right) \rtimes {{\mathbb{C}}_{2}} $ group, occur in leaf preserving diffeomorphism group and vector bundle isomorphism $(\xi, \eta)$ in Morse-Bott foliation on the solid Klein bottle \cite{MaksB} (because of matrix $A$ with $det(A) =-1$ change space orientation as on the Klein bottle), with the complementary circle.  


Group of continuous functions implementing rotation $D\left( y \right)$, which is a linear isomorphism preserving concentric circles, simultaneously with a shift as standing a second coordinate of tuple, is founded in \cite{MaksB} by S. I. Maksymenko. Its elements have a form of pair $(w{{e}^{2\pi i{{\lambda }_{h}}(s)}},s)$, where ${{\lambda }_{h}}(s)$ ensures sign inversion provided unit shift (on one).  We see that this group has structure of semidirect  product and denote it by $H$.
Thus, from this group $H$ of diffeomorphisms with additional functions
${{\lambda }_{h}}\left( s+1 \right)=-{{\lambda }_{h}}\left( s \right)$ making changing of sign provided by action of shift on one described in \cite{MaksB} homomorphism in subgroup of $ES{{L}_{2}}(\mathbb{R})$ can be constructed.    Homomorphic image can be realized by matrices of rotation with sign inversion inducing by the top group of semidirect product $ES{{L}_{2}}(\mathbb{R})$ that could be also generated by Frobenius normal form
$\left( \begin{array}{rr}
   0\,\, &-1 \\
  1\,\,& 0 \\
\end{array} \right)$.
Thus this subgroup of $ES{{L}_{2}}(\mathbb{R})$ can be embedded in $H$ and this subgroup  is realized by matrices of rotation with sign inversion due to the top group of semidirect product $ES{{L}_{2}}(\mathbb{R})$.
One of subgroup of our new group $ES{{L}_{2}}(\mathbb{R})$ is embedded in $H$. This subgroup is
$  SO(2)   \ltimes \left\langle \left( \begin{array}{rr}
 -1\,\,\,& 0 \\
  0\,\,\, &1 \\
\end{array}  \right) \right\rangle  \simeq : O(2)  $.
We additionaly denote this subgroup by $\left\langle\rho \right\rangle\ltimes \left\langle i  \right\rangle$.

Previously, a definition of an extended symplectic group was formulated for instance in \cite{Dan}, in terms of this paper a group of extended group is described as group of symplectic matrices with $det(M)=\pm1$, and denoted by $ESL\left( 2,{{\mathbb{Z}}_{\overline{d}}} \right)$ on page 4. But its structure was not found \cite{Dan, APPLEBY, Busch}.

We define it as the group of \emph{symplectic matrices} with $\det \left( M \right)=\pm 1$ additionally \textbf{find structure of extended symplectic matrices} and propose more convenient and usual notification of this group.

\emph{Extended symplectic group} be denoted by $ESp_2(\mathbb{R})$ is the group all symplectic matrices  having determinant $\det \left( M \right)=\pm 1$.
Thereby, \textbf{extended symplectic group} is subgroup of our group $ES{{L}_{2}}(\mathbb{R})$ and has the structure of semedirect product $ES{{p}_{2}}(\mathbb{R})\,\,\equiv \,\,S{{p}_{2}}(\mathbb{R})\rtimes {\mathbb{C}_{2}}$, where ${\mathbb{C}_{2}}$ is defined above, also symplectic group $S{{p}_{2}}(\mathbb{R})$ is the kernel of the semidirect product.  Note that ${\mathbb{C}_{2}}$ can be generated not only by $i$ but by matrix
$
\left( \begin{array}{rr}
1\,\,\,& 0 \\
0\,\,\, & -1 \\
\end{array} \right)
$
too. The justification of established structure is same as for $ES{{L}_{2}}(R)$.

As well known even symplectic  group has some applications \cite{APPLEBY, Busch}.

It is obvious that $ES{{p}_{2}}(R)\,\,<ES{{L}_{2}}\left( R \right)$.
We can spread concept of extended symplectic group on ring by considering $ES{{p}_{2}}(\mathbb{Z})$ and $ES{{p}_{2}}({{\mathbb{Z}}_{k}})$.
Then using finding by us structure \[ES{{p}_{2}}({{\mathbb{Z}}_{\overline{d}}})\,\, \simeq \,\,S{{p}_{2}}({{\mathbb{Z}}_{\overline{d}}})\rtimes {\mathbb{C}_{2}}\]
we can establish the structure of extended Clifford group more precisely  and apply it in theorem 2 \cite{APPLEBY} to describe a unique surjective homomorphism from extended Clifford group to group of Clifford operations which was used in \cite{Dan} in following homomorphism
${{f}_{E}}:\,\,\,\left( S{{p}_{2}}({{\mathbb{Z}}_{\overline{d}}})\rtimes {\mathbb{C}_{2}}  \right)\ltimes {{\left( {{\mathbb{Z}}_{\overline{d}}} \right)}^{2}}\to {}^{EC\left( d \right)}/{}_{I\left( d \right)}$
satisfying condition (110) from \cite{APPLEBY}.

In terms and notation of D. M. Appleby \cite{APPLEBY}, taking into consideration established here structure of $ESL(2,\mathbb{Z})$, the Clifford group from Theorem 2 takes form: $\left( SL\left( 2,{{\mathbb{Z}}_{\overline{d}}}\right) \rtimes {\mathbb{C}_{2}} \right)\ltimes {{\left( {{\mathbb{Z}}_{\overline{d}}} \right)}^{2}}$ wherein condition (110) from \cite{APPLEBY} holds.


 Note that group of the diffeomorphisms $h$ coinciding with some vector bundle morphism also function ${{\lambda }_{h}}:\,\,\mathbb{R}\to \mathbb{R}$ is described in item 3) of \cite{MaksB}, there are subgroup $h'(w,s)=\left( {{e}^{2\pi i{{\lambda }_{h}}(s)}},s \right),\,\,\,\,\,{{\lambda }_{h}}\left( s+1 \right)=-{{\lambda }_{h}}\left( s \right)$ presented in form of functions.  Now we can describe its structure as semidirect product. We establish a homomorphism from this group to $\left\langle\rho \right\rangle\ltimes \left\langle i  \right\rangle$.
    Furthermore the top group of $ES{{L}_{2}}(\mathbb{R})$ is the same matrix
$i= \left( \begin{array}{rr}
 -1\,\,\,& 0 \\
  0\,\,\, &1 \\
\end{array}  \right)$ coinciding with a matrix $\Lambda $ presenting the meridian of torus respect to the parallel \cite{MaksTor}.

The subgroup
of diffeomorphism $D\left( {{L}_{p,q}} \right)$ of $L_{p,q}$ is under consideration in \cite{MaksTor},
 whence a group closure of $D\left( {{L}_{p,q}} \right)$ is just $ES{{L}_{2}}(\mathbb{Z})$ but algebraic structure of set was not investigated before so it was classified in \cite{MaksTor}
  as the matrix subset of $G{{L}_{2}}\left( \mathbb{Z} \right)$  with determinant -1 also there is transformation $T$  in that item with $\det (T)=1$.

Thus, there are many subgroup of $ES{{L}_{2}}(\mathbb{Z})$ and whole $ES{{L}_{2}}(\mathbb{Z})$ appear in nature but it was not defined and investigated as algebraic group before.

In the model of rotations in the knee joint between the thigh and shin, which form the knee joint while being on opposite sides of the secant plane passing through the joint. Thus, the surfaces of the thigh and lower leg are on opposite sides of the cutting plane passing through the knee joint. Therefore, to specify a rotation operator in a single basis, you need exactly the operator represented by a matrix from the $ESL_2(\mathbb{R})$ group. By the same reason operators from our group can be applied in geoinformation systems \cite{Iat}.

Let $A=\left( \begin{array}{rr}
   0\,\, &-1 \\
  1\,\,\,& 0
\end{array} \right)$ so ${{B}_{1}}=\sqrt[{}]{A}=\frac{1}{\sqrt[{}]{2}}\left( \begin{array}{rr}
   2\,\,\,\, & - 1 \\
  1\,\,\,\,\,\, &2
\end{array} \right),$  ${{B}_{2}}=\sqrt[{}]{A}=\frac{1}{\sqrt[{}]{2}}\left( \begin{array}{rr}
   1\,\,\,\,&1 \\
  -1\,\,\,\,\,& 1
\end{array} \right).$  Thus we can present their product as a factorization of a matrix \[2{{B}_{1}}{{B}_{2}}=\left( \begin{array}{rr}
   2\, \,\,\,\,&-1 \\
  1\,\,\,\,\,\,& 2 \\
\end{array} \right)\left( \begin{array}{rr}
 1\,\,\,\,\,\,\, & 1 \\
 -1\,\,\,\,\,\,  & 1
\end{array} \right)=\left( \begin{array}{rr}
 \,\,\, 3\,\,\, & 1 \\
  -1\,\, \, & 3
\end{array} \right). \]
Besides this new method of matrix factorization \emph{\textbf{due to our square root existence criterions}} can be provided.
If $M$ possesses the presentation  $M=A-C$, where $A={{B}^{2}},\,C={{D}^{2}}$, then $M$ can be factorized in the following way $M=\left( B-D \right)\left( B+D \right)$. Therefore it is important to have quick method of square root existence checking in $S{{L}_{2}}\left( \text{F} \right)$. Analogously if $M$ admits  the presentation  $M=A-C$, where $A={{B}^{3}},\,C={{D}^{3}}$, then a factorization of $M$ is possible due to our formulas presented below.



\section {Criterion of an element root existing in $GL_2(\mathbb{F}_p)$, $SL_2(\mathbb{F}_p)$ and its formulas}

\subsection{Conditions of root existing in group and overgroup}
Let $SL_2(\mathbb{F}_p)$  denotes the special  linear group of degree 2 over a finite field of order $p$. And a degree always means an irreducible character degree in this paper.

We recall the well known relation between eigenvalues of $A$ and $f(A)$.
\begin{lem}\label{eigenval}
If $\beta $ is an eigenvalue for $B$ then ${{\beta }^{2}}$ is an eigenvalue for ${{B}^{2}}$.
\end{lem}

Consider the criterion of squareness of elements in $SL_2(\mathbb{F}_p)$ as well as in $GL_2(\mathbb{F}_p)$ which can be presented by diagonal matrix. As well known \cite{LinearShpringer} a matrix can be presented in the diagonal form iff the algebraic multiplicity of its eigenvalues are the same  as the geometric multiplicity.




\begin{thm}\label{CriterionPSL}
Let $A$ be simple or scalar matrix and
$A \in SL_2(\mathbb{F})$  \cite{LinearShpringer}, then
for  $A $ there is a solution  $B \in SL_2(\mathbb{F})$ of the matrix equation
 \begin{equation}\label{Eq}
X^2 = A
\end{equation}
if and only if
\begin{equation}\label{tr}  {\tr A+2}
\end{equation}

 is quadratic element in $\mathbb{F}$ or 0, where $\mathbb{F}$ is a field.


If  $X \in ESL_2(\mathbb{F})$ then the matrix equation  \eqref{Eq} has a solutions iff
\begin{equation}\label{trESL}
{\tr A \pm 2}
\end{equation}

 is quadratic element in $\mathbb{F}$ or 0.

This solution $X \in ESL_2(\mathbb{F}) \setminus SL_2(\mathbb{F})$ iff $\left({tr A - 2}{}\right)$ is quadratic element or 0 in $\mathbb{F}$ but $\left({\tr A + 2}{}\right)$ is not. Conversely $X \in SL_2(\mathbb{F})$ iff $\left({\tr A + 2}{}\right)$ is quadratic element. Solutions belong to $ ESL_2(\mathbb{F})$ and $SL_2(\mathbb{F})$ iff $\left({\tr A + 2}{}\right)$ and $\left({\tr A - 2}{}\right)$ are quadratic elements.



In the case $A \in GL_2 ({{\mathbb{F}}_{}})$ this condition \eqref{tr} takes form:
\begin{equation}\label{trGL}
 {\tr A \pm 2\sqrt{\det A}}
\end{equation}
 is quadratic element in $\mathbb{F}$ or 0 and $\det A$ is quadratic element.


\end{thm}

\begin{proof}
Throughout the proof a quadraticity of element $x$ or $x=0$ in a field $\mathbb{F}$ be denoted by $\left(\frac{x}{p}\right) \in \{0,1\}$. For concretization, we provide a proof over $\mathbb{F}_p$.
But out prove can be spread without changes on \emph{arbitrary field} $\mathbb{F}$ instead $\mathbb{F}_p$.

We assume that matrices $A$ and $B$ have eigenvalues $\lambda_1, \lambda_2 $ and $\mu_1, \mu_2 $ respectively.
Let a characteristic polynomial ${{\chi }_{B}}(x)$ of $B$ be the following: ${{\chi }_{B}}(x) = (x - \mu_1)(x - \mu_2)$. We denote $tr(A)$ by $a$.

 Since  $det(A),  A \in SL_n(\F_p)$ is 1, then eigenvalues of $A$ satisfy the following equality: $\mu_1^2\mu_2^2 = 1 $ that implies $\mu_1\mu_2 = \pm 1$. Therefore $a + 2\mu_1\mu_2 = a \pm2 = (\mu_1+\mu_2)^2$. As is known $tr(B)=\mu_1+\mu_2 \in  \F_p$ and $det(B)=\mu_1 \mu_2 \in \F_p$.
 Then according to Lemma \ref{eigenval} $a$ is the sum of the roots $\mu_1^2$, $\mu_2^2$ of a polynomial ${{\chi }_{A}}(x)=(x-\mu_1^2)(x-\mu_2^2)$. Hence $tr(A)= a = \mu_1^2 + \mu_2^2 = (\mu_1 + \mu_2)^2 - 2\mu_1\mu_2 = (tr(B))^2-2$.
 So,
${tr}(A)+2=c^2$ for $c={tr}(B)$.

In  case $\mu_1\mu_2 = - 1$ we express $tr(A)$  as $tr(A)= a = \mu_1^2 + \mu_2^2 = (\mu_1 - \mu_2)^2 - 2\mu_1\mu_2 = (tr(B))^2+2$  and conclude that ${tr}(A)-2=c^2$  is quadratic residue in this case.  It  yields that the solutions  $\pm B \in   ESL_2(\mathbb{F}) \setminus SL_2(\mathbb{F})$.


We show the existence of ${{\chi }_{B}}(x):={{x}^{2}}-cx+1$ having roots ${{\mu }_{1}},\,\,{{\mu }_{2}}$ which will be the e.v. of $B$.
Let ${{\chi }_{{{B}^{2}}}}(x)={{\mu }^{2}}-a\mu +1$. Then $\mu _{1}^{2},\,\,\,\mu _{2}^{2}$ are e.v. for $A$ and according to Viet's theorem, $\mu _{1}^{2}+\mu _{2}^{2}=a.$

Let us prove the sufficiency of the condition $(\frac{\tr A+2}{p})=1$.
According to Viet Theorem ${{\mu }_{1}}+{{\mu }_{2}}=c$ and ${{\mu }_{1}}+{{\mu }_{2}}=Tr(B)$, also ${{c}^{2}}=\tr A+2$ by construction of ${{\chi }_{B}}(x)$.

We assume that ${{\chi }_{B}}\left( x \right):={{x}^{2}}-cx+1=\left( x-{{\mu }_{1}} \right)\left( x-{{\mu }_{2}} \right)$, where $c:=\pm \sqrt{tr\left( A \right)+2}$, is characteristic polynomial for $B$  and ${{\chi }_{A}}\left( x \right):={{x}^{2}}-ax+1=\left( x-{{\lambda }_{1}} \right)\left( x-{{\lambda }_{2}} \right)$, where $a=tr\left( A \right)$.
To provide justification that ${{\chi }_{B}}\left( x \right)$ is characteristic polynomial of $\sqrt{A}$,  which denoted by$B$,   we consider ${{\chi }_{{{B}^{2}}}}\left( x \right)=\left( x-\mu _{1}^{2} \right)\left( x-\mu _{2}^{2} \right)$ and prove that ${{\chi }_{{{B}^{2}}}}\left( x \right)={{\chi }_{A}}\left( x \right)$ by showing coinciding of their coefficients. For this goal we have constructed  ${{c}^{2}}:=tr\left( A \right)+2$,  in another hand $c={{\mu }_{1}}+{{\mu }_{2}}$ and by condition of theorem $tr\left( A \right)+2$ is quadratic residue or 0.
Consider the sum $\mu _{1}^{2}+\mu _{2}^{2}={{\left( {{\mu }_{1}}+{{\mu }_{2}} \right)}^{2}}-2{{\mu }_{1}}{{\mu }_{2}}={{c}^{2}}+2-2{{\mu }_{1}}{{\mu }_{2}}={{c}^{2}}+2-2=\tr A=a$, according to Viet theorem $\mu _{1}^{2}+\mu _{2}^{2}$ is coefficient of linear term in ${{\chi }_{{{B}^{2}}}}$. The free term of ${{\chi }_{{{B}^{2}}}}\left( x \right)$ as well as of ${{\chi }_{A}}\left( x \right)$ equals to 1 as products of e.v. $\mu _{1}^{2}\mu _{2}^{2}=Det\left( B ^2 \right)$ and ${{\lambda }_{1}}{{\lambda }_{2}}=1$ because of ${{B}^{2}},  A\in S{{L}_{2}}\left( \mathbf{F} \right)$.  Thus coefficients of ${{\chi }_{{{B}^{2}}}}\left( x \right)$ and ${{\chi }_{A}}\left( x \right)$ coincide providing an equality of these polynomials.
So, their eigenvalues are the
same too. Also these eigenvalues are different. Hence these matrices are conjugated.

For the case of generalization on $GL_2(\F_p)$ the proof is the similar but with new absolute term in ${{\chi }_{B}}$.
Let $\det A=D$ and $D={{d}^{2}}$  if $\tr A+2\sqrt{\det A}$ is quadratic element then we construct ${{\chi }_{B}}\left( x \right)={{x}^{2}}-cx+d$, with $d=\pm \sqrt{D}$,  then ${{d}^{2}}=\mu _{1}^{2}\mu _{2}^{2}$, where ${{\mu }_{1}},\,\,{{\mu }_{2}}$ are e.v. of $B$. Consequently ${{\chi }_{{{B}^{2}}}}\left( x \right)={{x}^{2}}-\left( {{c}^{2}}-2 \right)x+{{d}^{2}}$ in the same time ${{\chi }_{A}}\left( x \right)={{x}^{2}}-tr\left( A \right)x+\det \left( A \right)$. Thus, these polynomials have the same coefficients, as in case of $S{{L}_{2}}\left( {{F}_{p}} \right)$. So ${{B}^{2}}$ and $A$ are conjugated matrices.

Consider the case of scalar matrix in $\mathbb{G}{{L}_{2}}\left( {{\text{F}}_{p}} \right)$.
Show that a characteristic polynomial also exists, in view of $c=\tr A-2\sqrt[{}]{\det A}=2\lambda -2\sqrt[{}]{{{\lambda }^{2}}}=2\lambda \pm 2\lambda.$ That is equal to
$$
2\lambda \pm 2\lambda =
\left[
\begin{array}{rr}
0 \,\,\,\,\,\, \mbox{iff} &  \,\, \sqrt[{}]{\det A}=-\lambda,\\
\,4\lambda \,\,\,\, \mbox{iff} &  \,\, \sqrt[{}]{\det A}=\lambda.
\end{array}
\right.
$$

The value $4\lambda =\tr A+2\sqrt[{}]{\det A}$ is declaimed in the condition \eqref{trGL} as quadratic residue, therefore $4\lambda \in {\text{F}}_{p}$. Also absolute term $d$ is  $\sqrt[{}]{\det A}=\sqrt[{}]{{{\lambda }^{2}}}=\pm \lambda \in \F_p $
because of  both elements $\lambda$  on diagonal and rest of elements is 0 moreover all conjugated matrices to a scalar matrix $A$ coincide with $A$ because $A$ in centre, that’s why $\lambda \in \F_p$.
Thus, the coefficients $c,$ $d \in {\text{F}}_{p}$, so such $B$ exists in $SL_2({\text{F}}_{p})$. In case $A \in \mathbb{S}{{L}_{2}}\left( {{\text{F}}_{p}} \right)$  our expression takes form  $\tr A-2\sqrt[{}]{\det A}=2 \pm2$ and its values are always squares.

In the case of diagonal matrix which is not scalar (case of simple matrix) we get $d =\pm \sqrt{\det A}$ but under additional condition to \eqref{trGL} $\det A$ is quadratic residue, hence we have $\pm \sqrt[{}]{\det A} \in \F_p$.


The structure of matrix roots $B_i$ of \emph{\textbf{exceptional limiting case}}, when  $\tr A+2=0$ corresponds to a scalar matrix $A= -E$  in   $S{{L}_{2}}\left( \mathbb{F} \right)$,
 then
$$
B_1=
\left(
\begin{array}{rr}
\pm\lambda & 0\\
0 & \pm \lambda
\end{array}
\right)
, \
B_2=
\left(
\begin{array}{rr}
0 & \pm\lambda\\
\mp\lambda & 0
\end{array}
\right)
, \
B_3=
\left(
\begin{array}{rr}
0 & 1\\
\lambda & 0
\end{array}
\right), \,
B_4=
\left(
\begin{array}{rr}
0 & \lambda\\
1 & 0
\end{array}
\right)
,$$
where  $\lambda^2=-1$.  It is obviously that this root exists if -1 is quadratic element in ${\mathbb{F}}_{}$, whence we see $ B_1, B_2$ are elements of $ESL_2(\mathbb{F}_p)$. If  $\tr A-2=0$ then we construct the same roots  but with condition  $\lambda^2=1$.

An outstanding case provided by $\lambda^2=1$ is Jordan form $J_A=
\left(
\begin{array}{rr}
\lambda & 1\\
0  & \lambda
\end{array}
\right)
,$
possess the solutions  $S_1=
\left(
\begin{array}{rr}
\pm 1 & \frac{1}{\pm 2}\\
    0 & \pm 1
\end{array}
\right)
$ from $ S{{L}_{2}}\left( \mathbb{F} \right)$ and

$ G_1=
\left(
\begin{array}{rr}
\pm \sqrt{\lambda} & \frac{1}{\pm 2 \sqrt{\lambda}}\\
    0 & \pm \sqrt{\lambda}
\end{array}
\right)
$ belonging to $GL_{2} \left( \mathbb{F} \right)$.

If  $A \in G{{L}_{2}}\left( \mathbb{F} \right)$ and satisfies \eqref{trGL}
then the case  $\tr A-2\sqrt{\det A}=0$,  where  $A= \lambda E$  implies that  $\tr A =2 \lambda$,  and its  roots
$$
\sqrt{A}=
\left(
\begin{array}{rr}
\pm\lambda & 0\\
0 & \pm \lambda
\end{array}
\right)
, \,
\sqrt{A}=
\left(
\begin{array}{rr}
0 & \pm\lambda\\
\mp\lambda & 0
\end{array}
\right)
, \,
\sqrt{A}=
\left(
\begin{array}{rr}
0 & 1\\
\lambda & 0
\end{array}
\right)
, \, \sqrt{A}=
\left(
\begin{array}{rr}
0 & \lambda\\
1 & 0
\end{array}
\right)$$
where  $\lambda^2=1$.  Note all roots are conjugated in view of scalar structure of  $A$.

The case  $\tr A-2=0$  implies that  $A= E$ , so its  roots
$$
\sqrt{A}=
\left(
\begin{array}{rr}
\pm\lambda & 0\\
0 & \pm \lambda
\end{array}
\right)
, \
\sqrt{A}=
\left(
\begin{array}{rr}
0 & \pm\lambda\\
\mp\lambda & 0
\end{array}
\right)
, \
\sqrt{A}=
\left(
\begin{array}{rr}
0 & 1\\
\lambda & 0
\end{array}
\right)
,$$
where  $\lambda^2=1$.

The
sequence of e.v., corresponding to the limiting case ($\underset{{{\lambda }_{i}}\to 1}{\mathop{\lim }}\,\left( Tr{{A}_{i}}+2 \right)=0$), is ${{\lambda }_{i}}+\frac{1}{{{\lambda }_{i}}}\to 2$. In this sequence matrices are simple and have diagonal form as
well as their roots have limiting form. But the limiting case admits not diagonal structures
of roots, where  all roots are conjugated
i.e. similar matrix. Indeed if  $A'$ and $A$ are similar matrix and ${{\left( B'
\right)}^{2}}=A'$  then  ${U^{-1}}A'U={U^{-1}}{{(B')}^{2}}U={{U}^{-1}}B'U{{U}^{-1}}B'U={{B}^{2}}=A$
so $B={{U}^{-1}}B'U$.






Let us construct the solution of equation $X^2=A$ in $SL_2(\F_p)$. In a general case we obtain the solution \[{{B}^{2}}=A',\] where $A'\sim A$ with eigenvalues ${{\lambda }_{1}}=\mu _{1}^{2},\,\,\,{{\lambda }_{2}}=\mu _{2}^{2}.$     Since $c \in {{\text{F}}_{p}}$ then we can construct in normal Frobenius form a matrix
\[\left( \begin{array}{rr}
  0\,\,& -1 \,  \\
  1\,\, &c \, \\
\end{array} \right)={{B}_{}}\]

 therefore this matrix is over base field ${{\text{F}}_{p}}$ or ${{\text{Q}}}$ or arbitrary field $\mathbb{F}$.
Since ${{\lambda }_{1}}+{{\lambda }_{2}}={{(\mu _{1}^{}+\mu _{2}^{})}^{2}}-2=\tr A$ and  that is why ${{\left( \mu _{1}^{{}}+\,\mu _{2}^{{}} \right)}^{2}}=\tr A+2$ this equality holds iff  $(\frac{\tr A+2}{p})=1$. Thus, the condition $(\frac{\tr A+2}{p})=1$ is  sufficient for existing of ${{\chi }_{b}}(x)$.
But it remains to show that these eigenvalues $\sqrt{\lambda_1}=\mu_1$, $\sqrt{\lambda_2}=\mu_2$ are the roots of the characteristic polynomial $\chi_B(x)$.

By the condition of theorem $\tr A+2$ is a quadratic residue or 0, there is $\sqrt{tr(A)+2}=\sqrt{{{({{\mu }_{1}}+{{\mu }_{2}})}^{2}}}$ in ${{\text{F}}_{\text{p}}}$, whence $tr(\text{B})\in {{\text{F}}_{\text{p}}},\,\,\det B\in {{\text{F}}_{\text{p}}}$ holds in view of well known theorems, therefore ${{\chi }_{B}}(x)$ has coefficients $c=\sqrt{tr(A)+2} = {\mu }_{1}+{\mu }_{2}$
 in ${{\text{F}}_{\text{p}}}$, hence $B$ presented in the Frobenius normal form belongs to $S{{L}_{2}}({{\text{F}}_{\text{p}}})$.

  Furthermore $B$  having  e.v.  ${{\mu }_{1}},\,\,{{\mu }_{2}}$ is the matrix over ${{\text{F}}_{\text{p}}}$, but ${{\mu }_{1}},\,\,{{\mu }_{2}}$ can be from ${{\text{F}}_{{{\text{p}}^{2}}}}\backslash {{\text{F}}_{\text{p}}}$.
\end{proof}

\begin{exm}
    Consider Fibonacci matrix $F=\left( \begin{array}{rr}
        0 &  \,\, 1  \\
        1 & \,\,  1
    \end{array}
   \right)$ in $S{{L}_{2}}\left( {\F_{p}} \right)$ then ${F^{2}}=\left( \begin{array}{rr}
      1  &  \,\,1\\
       1 & \,\,2
   \end{array}
   \right)$ which confirms criterion \ref{CriterionPSL} of existing roots in $ES{{L}_{2}}\left( { \mathbb{F}_{p}} \right)$ because $\tr A-2=1$ because of 1 is square in each field ${\mathbb{F}_{p}}$ as well as in $\mathbb{Q}$ and $\mathbb{R}$.

Next one is
$R=\left( \begin{array}{rr}
      0  &  \,\,-2\\
       2 & \,\, 0
   \end{array}
   \right)$
 in $S{{L}_{2}}\left( {\mathbb{F}_{3}} \right)$ then
 $R^2=\left( \begin{array}{rr}
      -1  &  \,\,0\\
       0 & \,\, -1
   \end{array}
   \right)$.

  In another hand we can justify the root existing by criterion for $ES{{L}_{2}}\left( {\mathbb{F}_{3}} \right)$ because $tr{{R}^{2}}-2=0$.
\end{exm}

\begin{exm}
The case of roots belonging to both cosets of quotient ${}^{ES{{L}_{2}}\left( \mathbb{F} \right)}/{}_{S{{L}_{2}}\left( \mathbb{F} \right)}$ appears, in particular, for matrix $A$ with $tr(A)=3$ and $\mathbb{F}=\mathbb{F}_{11}$. In fact, in this case $tr(A)-2=1$, $tr(A)+2=5$ one easily can verify that 5 is quadratic residue by $\bmod11$ because of ${{4}^{2}}\equiv 5\left( \bmod 11 \right)$ and  $1$ is always square.
\end{exm}

\begin{exm}
    Consider a case when  roots are only from $ESL_2(\mathbb{ Z})$, let $A=\left( \begin{matrix}
   3 & 2  \\
   4 & 3  \\
\end{matrix} \right)$. Here  $\tr A-2$ is square.
$\tr A -2\text{ }=\text{ }4$  but $\tr A+2\text{ }=\text{ 8}$ is not square in $\mathbb{ Z}$. The square roots

\[B=~\frac{\pm 1}{\sqrt{4}}\left( \begin{matrix}
  2\,\,2 \\
  4\,\,2 \\
\end{matrix} \right)~= \pm  \left( \begin{matrix}
   1\,&\,1 \\
  2\,&\,1 \\
\end{matrix} \right), ~\]
therefore $B\in ESL_2(\mathbb{ Z})$.
\end{exm}

\begin{exm}
Consider matrix equation $X^2=A$ with e.v. in ${{\mathbb{F}}_{9}}\backslash {{\mathbb{F}}_{3}}$ and having root in $ES{{L}_{2}}\left( {{\mathbb{F}}_{3}} \right)\backslash S{{L}_{2}}\left( {{\mathbb{F}}_{3}} \right)$ and check our new formula from Proposition \ref{formula} (which is directly below) for root expression. Let $A=\left( \begin{array}{rr}
   0\,\,&-1 \\
  1\,\,\,& 0 \\
\end{array} \right)$

 since $\tr A+2=2$ that is non-square residue in ${{\mathbb{F}}_{3}}$ but $\tr A-2=-(-1)=1$ in ${{\mathbb{F}}_{3}}$, then  according to our formula we use branch of expression with minus in $\tr A \pm 2$ i.e.
$\sqrt{A}=\frac{A-E}{\sqrt[{}]{\tr A-2}}=\frac{1}{\sqrt[{}]{\tr A-2}}\left( \begin{array}{rr}
   -0-1,\,\,\,\,\, & -1 \\
  \,1,\,\,\,\,\,\,\,& -0-1
\end{array}
\right)=\frac{1}{\sqrt[{}]{\tr A-2}}\left( \begin{array}{rr}
   -1,\,\,\,\,\,&-1 \\
 \,1,\,\,\,\,\,&-1 \\
\end{array} \right)=\left( \begin{array}{rr}
   -1,\,\,\,\,\,&-1 \\
  \,1,\,\,\,\,\,&-1 \\
\end{array} \right)=\left( \begin{array}{rr}
   \,2,\,\,\,& -1 \\
  \,1,\,\,\,\,\,&2 \\
\end{array} \right) = B.$  Another branch with "-" before the root $\tr A \pm 2$ lead us to second root:
$\sqrt{A}= \frac{A-E}{- ~ \sqrt[{}]{\tr A-2}}=\frac{1}{\sqrt[{}]{\tr A-2}}\left( \begin{array}{rr}
   -0+1,\,\,\,\,\, & 1 \\
\end{array}
\right)=\frac{1}{\sqrt[{}]{\tr A-2}}\left( \begin{array}{rr}
    1,\,\,\,\,\,&1 \\
 \,-1,\,\,\,\,\,&1 \\
\end{array} \right)= -B.$

Its ${{\chi }_{A}}(x)={{x}^{2}}+1=0$ therefore its roots are $\pm i\in {{\mathbb{F}}_{9}}\backslash {{\mathbb{F}}_{3}}$ and $\pm i$ are square in ${{\mathbb{F}}_{9}}$ that confirms our criterion \ref{CriterionPSL}.
\end{exm}

\begin{cor} Let $A \in SL_2(\mathbb{F})$, where $\mathbb{F}$ is arbitrary field.
Then all solutions of equation $X^2=A$ contain in $ESL_2(\mathbb{F})$.
\end{cor}
\begin{proof} Let $B$ is solution of $X^2=A$.
Since $1=\det A= \det B \det B$ then $\det B=\pm1$. Thus, $B \in ESL_2(\mathbb{F})$.
\end{proof}

For case $\mathbb{F}=\mathbb{F}_p$ our criterion can be formulated in terms of Legendre symbol.
\begin{cor}\label{CriterionSL(2,p)}
Let $A$ be simple matrix and
$A \in {SL_2( \mathbb{F}_p)}$  \cite{LinearShpringer}, then
for matrix $A \in SL_2(\mathbb{F}_p)$ there is a solution  $B \in SL_2(\mathbb{F}_p)$ of the matrix equation
 \begin{equation}\label{Eq1}
X^2 = A
\end{equation}
if and only if

\begin{equation}\label{tr1}
\left(\frac{\tr A+2}{p}\right) \in \{0,1\}.
\end{equation}

If  $X \in ESL_2(\mathbb{F}_p)$ then the matrix equation  \eqref{Eq1} has a solution iff

\begin{equation}\label{tr2}
\left(\frac{\tr A \pm 2}{p}\right) \in \{0,1\}.
\end{equation}
This solution $X \in ESL_2(\mathbb{F}_p) \setminus SL_2(\mathbb{F}_p)$ iff $\left(\frac{tr A - 2}{p}\right)=1$ or 0, but $\left(\frac{\tr A + 2}{p}\right)=-1$. Conversely $X \in SL_2(\mathbb{F}_p)$ iff $\left(\frac{\tr A + 2}{p}\right)=1$.
Solutions $X_{i} \in ESL_2(\mathbb{F})$ and $SL_2(\mathbb{F})$ iff $\left(\frac{\tr A + 2}{p}\right)=1$ and $\left({\tr A - 2}{}\right)=1$.

In the case  $A \in GL_2 ({{\mathbb{F}}_{p}})$ this condition \eqref{tr} takes form:
 \begin{equation} \label{trGL(p)}
\left(\frac{\tr A \pm 2\sqrt{detA}}{p} \right) \in \{0,1\}.
 \end{equation}
\end{cor}
The proof is the same but instead of $\mathbb{F}$ we put $\mathbb{F}_p$. But we emphasize that theorems of such a kind were for algebraic closed field before this paper.

\begin{cor}
If $A\in GL(F_2)$ the condition \ref{tr} takes the form:
\[\left(\frac{\tr A}{p}\right) \in \{0,1\}. \]
\end{cor}

\begin{rem}
    The formulated criterion for a diagonizable matrix  is also true over fields $\mathbb{Q}$ and $\mathbb{R}$.
\end{rem}
\begin{proof}
The proof is the same only with the change of quadraticity criterion over the new field.
\end{proof}

\begin{cor} If matrix $A$ admits  diagonal form over ${\F}_2$ then $A$ is square over ${\F}_{2}$.
\end{cor}
\begin{proof}
Since in $F_{2^n}$ all elements $g_i \in F_{2^n}$ are quadratic elements, therefore a diagonal matrix $A$ is always square of the mentioned above $B$ over $F_{2^n}$.
\end{proof}

We revise the formula of square root for its generalization and also because of a limiting case of zero in the denominator was not researched fully. In fact, a root admitting Jordan block of dimension 2 was not found in \cite{DON}.
\begin{prop}\label{formula}
If a simple matrix  $A\in S{{L}_{2}}({\mathbb{F}_{p}})$ and  $\left( \frac{\tr (A)+2}{p} \right)=1$, then

$$\sqrt{A}=\frac{ 1}{\pm \sqrt{ \tr A \pm2}}\left( A \pm E \right),$$
where $E$ is identity element of $S{{L}_{2}}({\mathbb{F}_{p}})$, in case of sign '-' in $\left( A \pm E \right)$  roots  $\sqrt{A} \in ESL_{2}({\mathbb{F}_{p}})$.  Namely for $A= \left( \begin{array}{cc}
    a & b \\
     c & d \\
  \end{array} \right)$ in a coordinate form in case $\sqrt{A} \in SL_2({\mathbb{F}_{p}})$ we have
$
 {\sqrt{A}}= \frac{1}{\pm \sqrt{\tr (A)+2}} \left(
  \begin{array}{cc}
    a+1 & b \\
     c & d+1 \\
  \end{array}
\right)$.


\end{prop}

\begin{proof}

Consider the characteristic equation for $A$:  ${{x}^{2}}-\tr(A)x+\det (A) E=0$.  According to Cayley Hamilton theorem we have
\begin{align*}
  & {{A}^{2}}-\tr (A)A+E=0,  \\
  & {{A}^{2}}+E=\tr (A)A.
\end{align*}
  We add $2A$ to the both sides of this equation
\begin{align*}
  & {{A}^{2}}+2A+E=-\tr (A)A+2A, \\
  & {{\left( A+E \right)}^{2}}=A\left( \tr A+2 \right).
\end{align*}
And finally  we express the root:
\begin{equation}\label{sq}
 \sqrt{A}=\frac{1}{\pm \sqrt{\tr A+2}}\left( A+E \right).
\end{equation}


In a coordinate it takes the form:
$
 {\sqrt{A}}= \frac{1}{\pm \sqrt{\tr A+2}} \left(
  \begin{array}{cc}
    a+1 & b \\
     c & d+1 \\
  \end{array}
\right).$

The similar proof for the case ${\tr A-2}$ is square give us roots in $ ESL_2(\mathbb{\F_p})$.

\begin{equation}\label{esq}
 \sqrt{A}=\frac{1}{\pm \sqrt{\tr A-2}}\left( A - E \right).
\end{equation}

In a coordinate form it presents by matrix:

$
 {\sqrt{A}}= \frac{1}{\pm \sqrt{\tr A-2}} \left(
  \begin{array}{cc}
    a-1 & b \\
     c & d-1 \\
  \end{array}
\right).$
  \end{proof}

The similar proof yields the formula of roots in $ ESL_2(\mathbb{Z})$ and in $ ESL_2({k})$, where $k$ is arbitrary perfect field.

\begin{exm}
For instance, $tr (E)=2$ and as a result $
 {\sqrt{E}}=  \frac{1}{\pm \sqrt{2+2}}\left(
  \begin{array}{cc}
    2 & 0 \\
     0 & 2 \\
  \end{array}
\right) = \pm E.
$
\end{exm}

Note the in analytical formula of root in $GL_2 (\mathbb{R})$ founded in \cite{DON} the case of finite field $\mathbb{F}_p$ was not considered. Furthermore, case when $ \tr(A) + \epsilon_1 2 \sqrt{det(A)} E= 0$ was not provided not in their formula nor in their work \cite{Kunyav2}.
We also take this case into account in our investigations.

\begin{cor}\label{formulaGL}
 Generalizing the formula of root on $G{{L}_{2}}({\mathbb{F}_{p}})$ we get a new formula.
If a simple matrix  $A\in G{{L}_{2}}({\mathbb{F}_{p}})$ and $\left( \tr(A) \pm 2A \sqrt{det (A)}  \right)$ is quadratic residue or 0 in ${\mathbb{F}_{p}}$, then $\sqrt{A}\in G{{L}_{2}}({\mathbb{F}_{p}})$
$$\sqrt{A}=\frac{ \pm 1}{ \sqrt{\tr A \pm 2 \sqrt{det A}}}\left( A \pm E \sqrt{det A} \right),$$
where  sign $'\pm'$ in $\left( A \pm E \sqrt{det A} \right)$  coincides with  sign $'\pm'$  in denominator $\sqrt{\tr A \pm 2 \sqrt{det A}}$, whereas sign before 1 in the nominator is independent.
\end{cor}

\begin{proof}
Consider the characteristic equation for $A \in GL(\mathbb{F})$:  ${{x}^{2}}-\tr(A)x+\det (A) E=0$.  Then using Cayley Hamilton theorem we derive
\begin{align*}
  & {{A}^{2}}-\tr(A)A+det(A) E=0,  \\
  & {{A}^{2}}+E=\tr(A)A.
\end{align*}
  We add $\pm 2A\sqrt{detA}$ to form a complete square in left side of equation
\begin{align*}
  & {{A}^{2}}+2A(\pm \sqrt{detA})+E=-\tr(A)A+2A\sqrt{detA}, \\
  & {{\left( A \pm E{detA} \right)}^{2}}=A\left( \tr A+2\sqrt{detA} \right).
\end{align*}
This lead us to solution in similar way described above. But the exceptional limiting case $c=\tr A-2\sqrt[{}]{\det A}= 0$ was not founded in \cite{DON} so we describe it in the prove of Theorem \ref{CriterionPSL}, where we investigate two possible subcases in this situation $0= \tr A-2\sqrt[{}]{\det A} = 2\lambda -2\sqrt{{{\lambda }^{2}}} = 2\lambda \pm 2\lambda$ where in particular the solutions $B_1, B_2 \in ESL_2 (\mathbb{F})$ appear. Also roots of third and fourth power were not founded in \cite{DON, Nort, Arslan}. The investigation \cite{Nort} claims that there are class of matrices in $S{{L}_{2}}\left( {{\text{F}}_{2}} \right)$ having not square root but we make group classification of roots distribution in which root always exists in splotable group extension of $S{{L}_{2}}\left( {{\text{F}}_{p}} \right)$ by the same field viz it is in $ES{{L}_{2}}\left( {{\text{F}}_{p}} \right)$.
\end{proof}

\begin{cor}
    The formula of 4-th power root is the following
    $$\sqrt[4]{A}=\frac{A \pm E \pm\sqrt{\tr A \pm 2}}{\pm \sqrt{ \pm \sqrt{\tr A \pm 2} \pm 2}}.$$
\end{cor}
\begin{proof}
    We construct the formula to 4-th power root in recursive way where base of recursion in formula \eqref{sq}. Taking into account $tr(A+E)=\tr A+2$  we obtain for case $det(A)=det(\sqrt{A})=1$ that  $\sqrt[4]{A}=\frac{\sqrt{A}+E}{\pm\sqrt{tr\sqrt{A}+2}}=\frac{A+E \pm\sqrt{\tr A+2}}{\pm \sqrt{\sqrt{\tr A+2}+2}}.$
\end{proof}

\begin{rem}
    Extended special linear group $ESL_2(k)$, where $k$ is arbitrary perfect field, is storage of all square matrix roots from $SL_2(k)$.
\end{rem}

\begin{prop} If matrix $A$ do not admits diagonal form over ${\F}_{2}$ then $A$ is not square in $G{L_{2}}\left( {{\F}_{2}} \right)$ over ${\F_{2}}$.
\end{prop}
\begin{proof}
We consider equation of form ${{X}^{2}}=A$ and show that it has not solutions over ${{\F}_{2}}$ a $S{{L}_{2}}\left( {{\F}_{2}} \right)$ in case ${{\chi }_{A}}\left( x \right) \neq {{\mu }_{A}}\left( x \right)$. The conditions of theorem implies that geometrical dimension  of e.v. is  1 but algebraic multiplicity of e.v. $\lambda $ is 2. We make prof by the contradiction, assuming that is true then

${\left(\begin{array}{cc}
   \lambda \,\,& 1 \\
  0\,\, & \lambda  \\
\end{array} \right)}^{2}=\left(\begin{array}{cc}
   {{\lambda }^{2}}\,\, & 2\lambda  \\
  0 \,\, & {{\lambda }^{2}} \\
\end{array} \right)$  but
$\left( \begin{array}{cc}
   {{\lambda }^{2}}\,\,& 2\lambda \\
  0\,\, & {{\lambda}^{2}} \\
\end{array}\right)=
 \left(\begin{array}{cc}
   {{\lambda }^{2}}\,\,&0 \\
  0\,\,&{{\lambda }^{2}}\\
\end{array} \right)$ over ${{\text{F}}_{2}}$.
That contradicts to condition of this Theorem.
\end{proof}

Let $B$ has characteristic polynomial ${{x}^{2}}+bx+c$. It is well known that trace of $B$ is stable
under choosing of vector space base.

We denote Jordan form of matrix $A$ as ${{J}_{A}}$.

\begin{lem}\label{Trace} If a matrix $A\in S{{L}_{2}}({{F}_{p}})$ has multiple eigenvalues ${{\beta }_{1}}={{\beta }_{2}}=\beta $  and non-trivial Jourdan block of size $2\times 2$ then $\beta \in {{\F}_{p}}.$ 
\end{lem}

Proof.  Since in this case eigenvalues are presented as elements of matrix $B$  standing on diagonal, then this matrix can be in form: $B=\left( \begin{array}{cc}
   \beta \,\,\,& 1 \\
  0\,\, & \beta  \\
\end{array} \right)$   or
$B=\left( \begin{array}{cc}
   \beta \,\,\,& 0 \\
 0\,\,\,  & \beta  \\
\end{array} \right).$

But the eigenvalues  of the matrix are multiples, therefore $\beta +\beta =tr(B)\in {{\F}_{p}}$. This implies $2\beta =b$, therefore in a field of characteristic  non equal 2  we express this eigenvalue as $\beta =\frac{b}{2}$. Hence $\beta \in {{\F}_{p}}$.  The proof is completed.

Our study of quadratic elements in $SL_2(\mathbb{F})$ gives an instrument to solve the problem
when a finite group G contains a conjugacy class $K$ whose square $K^2$ is again a conjugacy class \cite{Belt}.

\begin{thm}\label{CriterionPSLJordan}
 Under conditions $(\frac{\lambda }{p})=1$ in ${{\F}_{p}}$ and matrix $A$ is similar to a Jordan block of the form
\begin{equation}\label{Jblock}
 {{J}_{A}}= \left(
  \begin{array}{cc}
    \lambda & 1 \\
     0 & \lambda \\
  \end{array}
\right)
\end{equation}
a square root $B$ of $A$ exists in $S{{L}_{2}}({{\F}_{p}})$.
\end{thm}

\begin{proof}
Assume that square from $A$ exists in $S{{L}_{2}}({{\F}_{p}})$  or in $ES{{L}_{2}}({{\F}_{p}})$ correspondently. We denote matrix $B$ transformed to upper triangular form by $U{{T}_{B}}$. Let us show that there that provided condition above it always exists such $B:\,\,UT_{B}^{2}={{J}_{A}}$, where $U{{T}_{B}}$ is $B$ transformed to UTM form.  Then we show that it implies existing  of solution of \[{{X}^{2}}=A. \]
From the existence of the Jordan block for $A$ follows the existence of a similarity transformation $U$ transforming ${{B}^{2}}$ to the Jordan normal form $J_B$ because of  $A = {{B}^{2}}$ and A has non-trivial Jordan block denoted by $J_A$. But square root from $B^2$ this operator $U$ transforms in upper triangular form $UT_B$. Then if we find solution for

\begin{equation}\label{uppertriang}
UT_B^2=J_A
\end{equation}

we can obtain solution for ${{X}^{2}}=A$ because of the following:

\begin{equation}\label{eqtransform}
A=U\cdot (U{{T}_{B}})^2\cdot {{U}^{-1}}=(U\cdot U{{T}_{B}}\cdot {{U}^{-1}})(U\cdot U{{T}_{B}}\cdot {{U}^{-1}})={{B}^{2}}.
\end{equation}

It means that such matrix $U{{T}_{B}}$ satisfying \eqref{eqtransform}, exists and it can be transformed  by the same similarity transformation by conjugation in form $U{{T}_{B}}={{U}^{-1}}BU$ by the same matrix that transforms $A$ in ${{J}_{A}}$ because of ${{B}^{2}}=A$. To show the existing of such solution of \eqref{uppertriang} we acting by invers transformation $A=U\cdot (U{{T}_{B}})^2\cdot {{U}^{-1}}=(U\cdot U{{T}_{B}}\cdot {{U}^{-1}})(U\cdot U{{T}_{B}}\cdot {{U}^{-1}})={{B}^{2}}$, where $U$ is similarity transformation $B$ to


$$U{{T}_{B}}=\left(
  \begin{array}{cc}
   \beta \,\,\, & \gamma  \\
  0\,\,\, & \beta  \\
  \end{array}
 \right).$$
note that its diagonal elements ${{b}_{11}}={{b}_{22}}=\beta $ are the same. Therefore according to Lemma \ref{Trace} we have $\beta \in {{F}_{p}}$. Even more easier we can deduce it without Lemma 19. We have ${{b}_{11}}={{b}_{22}}=\beta $, then  $\beta +\beta =\text{Tr}({{U}^{-1}}BU)$. Therefore $2\beta \in {{\text{F}}_{p}}$. It implies that $\beta \in {{F}_{p}}$ if $p>2$ and

$$(UT_{B})^{2}=\left( \begin{array}{cc}
  {{\beta }^{2}}\,\,\,  &  2\beta \gamma  \\
  0\,\,\,\,\,\,\,\,\,  & {{\beta }^{2}} \\
\end{array} \right). $$

Here the element $\gamma $ can be chosen $\gamma :\,\,\, 2\beta \gamma =1$ so $\gamma = {2 \beta}^{-1}$ taking into account that $\beta =\sqrt{\lambda }$ which is already determined by $A$. Then  $(UT_{B})^{2}:$

$(UT_{B})^{2}=\left( \begin{array}{cc}
   {{\beta }^{2}}\,\,\, & 2\beta \gamma  \\
  0\,\,\,\,\,\,\,\,\, & {{\beta }^{2}} \\
\end{array} \right)=\left( \begin{array}{cc}
   {{\beta }^{2}}\,\,\, & 1 \\
  0\,\,\,\,\, & {{\beta }^{2}} \\
\end{array} \right)=J_A = \left(
	\begin{array}{cc}
	\lambda \, & 1 \\
	0 & \lambda \, \\
	\end{array}
\right) $.

Furthermore we show that these conditions is also necessary but not only sufficient. It means if $(\frac{\lambda }{p})=-1$, then there are no matrix $B$ over $S{{L}_{2}}({{F}_{p}})$ such that ${{B}^{2}}=A$. By a reversal of theorem condition  and using the representation in the form of UTM for and for we see that $B$ from $PS{{L}_{2}}({{F}_{p}})$ such that ${{B}^{2}}=A$. We see that according to the Lemma \ref{Trace}
 the eigenvalue of $B$ over lie in the main field ---  ${{\text{F}}_{p}}$. However, we assumed that  $(\frac{\lambda }{p})=-1$. Thus we obtain the desirable  contradiction.

Let us show that condition of  non-diagonalizability of matrix is necessary in the conditions of this Theorem. By virtue of the well-known theorem stating that if the algebraic multiplicity is equal to the geometric  multiplicity for each eigenvalue, then matrix is diagonalizable otherwise it is not diagonalizable, we see that if the condition of  similarity  to  ${{J}_{A}}=\left( \begin{array}{cc}
   \lambda \,\,\, & 1 \\
  0\,\,\, & \lambda  \\
\end{array} \right)$
indicated in this Theorem \ref{CriterionPSLJordan}  does not holds, then  such $A$ satisfy the conditions of this Theorem \ref{CriterionPSL}, where algebraic multiplicity is equal to geometrical. And since the condition \ref{Jblock} of this criterion is nature, therefore, it is no longer necessary to prove the non-diagonalizability condition in Theorem \ref{CriterionPSLJordan}.

Proof of \textbf{necessity}.
Furthermore we show that these conditions is also necessary but not only sufficient. It means if $(\frac{\lambda }{p})=-1$, then there are no matrix $B$ having non trivial Jordan block over $S{{L}_{2}}({{\F}_{p}})$ such that ${{B}^{2}}=A$. By a reversal of theorem condition  and using the representation in the form of UTM for and for we see that $B$ from $S{{L}_{2}}({{\F}_{p}})$ such that ${{B}^{2}}=A$. We see that according to the Lemma the eigenvalue of $B\in S{{L}_{2}}({{\F}_{p}})$ correspondingly,  lie in the main field -- ${{\F}_{p}}$. Furthermore according to Lemma \ref{eigenval} if $\beta $ is an eigenvalue for $B$ then ${{\beta }^{2}}$ is an eigenvalue for ${{B}^{2}}$, so we have ${{\beta }^{2}}=\lambda $.  However, we assumed that $(\frac{\lambda }{p})=-1$. Thus we obtain the desirable  contradiction.    The eigenvalue $\beta$ has geometrical dimension 1, because of in oppositive case geometrical $\dim\beta =2$ (dimension of eigenvector space of $\beta$), then we get that $J_{B}^{2}$ is only scalar matrix $B$.

The proof is fully completed.
\end{proof}



\begin{exm}
A  sufficiency  of the condition $(\frac{\lambda }{p})=1$ in Theorem \ref{CriterionPSLJordan} for $\exists $ $B:\,\,{{B}^{2}}=A$, where $A \sim {{J}_{A}}$ of size $2\times 2$ with one eigenvalue corresponding to one eigenvector is given by following matrix from $S{{L}_{2}}(R)$:

${{J}_{A}}=\left( \begin {array}{cc}
   1\,\,\, & 1 \\
  0\,\,\, & 1 \\
\end{array} \right)$ then $B={{\left( \begin{array}{cc}
  \mu \,\,\,  & 1 \\
  0\,\,\,  & \mu  \\
\end{array} \right)}^{2}}=\left( \begin {array}{cc}
   {{\mu }^{2}}\,\,\, & 2\mu  \\
  0\,\,\,\,\,\,  & {{\mu }^{2}} \\
\end{array} \right),\,\,\,\,\mu =\pm \sqrt{1}$.

This confirms Theorem \ref{CriterionPSL}.
Choosing the base for $B$ to $A$ be in Jordan form (in Jordan base):  $UB{{U}^{-1}}$ we obtain

\[\left( \begin{array}{cc}
   \frac{\mu }{2}\,\,\,\, & 1 \\
  0\,\,\,\,\,\, & \frac{\mu }{2} \\
\end{array} \right)={{J}_{B}}.\]

The last matrix is expressed by conjugating of $B$ by a diagonal matrix.
\end{exm}

\begin{exm}
 Consider an example confirming Theorem \ref{CriterionPSL} .
 Let \\
$A=\left( \begin{array}{rr}
  -1\,\, & 0 \\
  0\,\, & -1
\end{array} \right)={{\rho }_{180}}$.  This is a $180$ degree rotation matrix.  The $Tr(A) + 2 = 0$  then root has to exist in $S{{L}_{2}}(R)$. Then its square root $B\in S{{L}_{2}}(R)$ has form ${{\rho }_{90}}=\left( \begin{array}{cc}
   \,0\,\,\,\,& 1 \\
  -1\,\,& 0 \\
\end{array} \right)=B\in S{{L}_{2}}(R)$.  Note that $A$ is presented in the diagonal form. There are also roots $B_{1}^{{}}=\left( \begin{array}{cc}
   i\,\,\,\,\,\,\, & a \\
  0\,\,\, & -i \\
\end{array} \right)$ from $S{{L}_{2}}(C)$ as well as $B_{2}^{{}}=\left( \begin{array}{cc}
   -i\,\,\,\,\,\, & a \\
  \,0\,\,\,\,\,\, & i \\
\end{array} \right)$.
\end{exm}

\begin{rem} \label{diagonal form and 2 non quadratic residue}
 If $A\in S{{L}_{2}}({{F}_{p}})$ possesses a presentation in diagonal Jordan form over  ${{\text{F}}_{p}}$ and $(\frac{{{\lambda }_{1}}}{p})=-1,\,\,(\frac{{{\lambda }_{2}}}{p})=1$,  then such case does not give the existence of solution of ${{X}^{2}}=A$ in $S{{L}_{2}}({{F}_{p}})$.
\end{rem}

\begin{proof}  The condition  $(\frac{{{\lambda }_{1}}}{p})=-1$ means, that $\sqrt{{{\lambda }_{1}}}={{\beta }_{1}}\in {{\text{F}}_{{{p}^{2}}}}\backslash {{\text{F}}_{p}}$  and simultaneously $\sqrt{{{\lambda }_{2}}}={{\beta }_{2}}\in {{\text{F}}_{p}}$,  therefore ${{\beta }_{1}}+{{\beta }_{2}}=Tr(B)\notin {{F}_{p}}$. This implies non-existing of ${{\mu }_{B}}\left( x \right)$ over ${{\text{F}}_{p}}$.
\end{proof}

The following theorem  it is true  for $S{{L}_{2}}({k})$, even  $k$ is arbitrary perfect field. The following proof works for arbitrary perfect ${F}$ too.
\begin{thm}
 If a matrix $A\in S{{L}_{2}}({{F}_{}})$ is semisimple and diagonalizable over ${{\text{F}}_{p}}$ and $(\frac{{{\lambda }_{1}}}{p})=(\frac{{{\lambda }_{2}}}{p})=-1$, then  for the existing $\sqrt{A}$, it is necessary and sufficient, to $A$ be similar to a scalar matrix $D$.
 \end{thm}

\begin{proof}
From the facts that $(\frac{{{\lambda }_{1}}}{p})=(\frac{{{\lambda }_{2}}}{p})=-1$ and the square of diagonal matrix is again the diagonal matrix follows the existence of root only in the off diagonal form, therefore we must find the solution $M$ among the set of non-diagonalizable matrices
$ D=\left( \begin{array}{cc}
    {{d}_{1}}\,\,\, & 0 \\
  0\,\,\, & {{d}_{2}} \\
 \end{array} \right) $
is the diagonal representation of matrix $A$, and let
 \begin {equation}\label{D}
 D={{M}^{2}},
\end {equation}
where $M\in S{{L}_{2}}({{F}_{p}})$. Because of $(\frac{{{d}_{1}}}{p})=(\frac{{{d}_{2}}}{p})=-1$ there is a  root in non-diagonal form. Also we note that there is a conjugation matrix $X$,

$$X=\left( \begin{array}{cc}
   m_{11}^{-1}\,\,\,& 0 \\
  0\,\,\,\, & m_{21}^{-1} \\
\end{array}  \right),$$  transforming  $M$  to $\tilde{M}$, where $\tilde{M}$ has following form

$$\tilde{M}=~\left( \begin{array}{cc}
   {{m}_{11}}\,\,\,&1 \\
  {{m}_{21}}\,\, &{{m}_{22}} \\
\end{array} \right). $$

Let's transform the equality $D={{M}^{2}}$ into $XD{{X}^{-1}}=XM{{X}^{-1}}XM{{X}^{-1}}$, where $XM{{X}^{-1}}=\tilde{M}$.
Note that $D$ and $XD{{X}^{-1}}$ have identical eigenvalues. Therefore we can solve the equation \eqref{D} for $XD{{X}^{-1}}$.  Let's consider matrix equation  $D={{M}^{2}}$,  let's transform it by conjugation
$D=XD{{X}^{-1}}=XM{{X}^{-1}}\cdot XM{{X}^{-1}}=\tilde{M}\tilde{M}={{\tilde{M}}^{2}}$,
wherein \\ $M=\left( \begin{array}{cc}
  & {{m}_{11}}\,\,\,\,{{m}_{12}} \\
 & {{m}_{21}}\,\,{{m}_{22}} \\
\end{array} \right)$, \,
$X=\left( \begin{array}{cc}
  & m_{11}^{-1}\,\,\,0 \\
 & 0\,\,\,\,\,\,m_{21}^{-1} \\
\end{array} \right)$
  and
$XM{{X}^{-1}}=\left( \begin{array}{ll}
   {{m}_{11}}\,\,\,\,& 1 \\
  {{m}_{21}}\,\, & {{m}_{22}} \\
\end{array} \right).$


 Since $D$ is a diagonal matrix, then it belongs to the commutative subgroup of diagonal matrices from  $S{{L}_{2}}({{F}_{p}})$, lets denote it as  $DS{{L}_{2}}({{F}_{p}})$.
Therefore and  $~XD{{X}^{-1}}$  is also a diagonal matrix. Moreover, due to the commutativity of the field ${{F}_{p}}$ we have $~XD{{X}^{-1}}=D.$
Now let's solve the matrix equation  for the reduced $\tilde{M}$

 \begin {equation}\label{XDX}
 D=~XD{{X}^{-1}}=(XM{{X}^{-1}})(XM{{X}^{-1}})={{\tilde{M}}^{2}},
\end {equation}

Note that equations \eqref{XDX} and \eqref{D} are equivalent since they are obtained by similarity transformations.

Note that equations (2) and (1) are equivalent since they are obtained by similarity transformations.
Let's write down the equation $$ {{\tilde{M}}^{2}} = {{\left( \begin{array}{cc}
  & {{m}_{11}}\,\,\,\,1 \\
 & {{m}_{21}}\,\,\,{{m}_{22}} \\
\end{array} \right)}^{2}}= \left( \begin{array}{cc}
  & {{d}_{1}}\,\,\,0 \\
 & 0\,\,\,{{d}_{2}} \\
\end{array} \right). $$

Thence we obtain the system of equations
$$\left\{ \begin{array}{cc}
  {{m}_{21}}+m_{11}^{2}={{d}_{1}} \\
 {{m}_{21}}+m_{22}^{2}={{d}_{2}} \\
 {{m}_{11}}+{{m}_{22}}=0, \\
\end{array} \right.$$
by substitution $m_{11}$ from the equation 3) ${{m}_{22}}=-{{m}_{11}}$ into equations 1) and 2) we express from \[2) \, \, {{m}_{21}}+m_{22}^{2}={{d}_{2}} \Rightarrow   {{m}_{21}}+{{(-m_{11}^{{}})}^{2}}={{d}_{2}}\]


also we take into consideration equation 1) $m_{11}^{2}+{{m}_{21}}={{d}_{1}}$. Thence ${{d}_{1}}={{d}_{2}}$ or more conveniently $d={{d}_{1}}={{d}_{2}}$. Wherein $d$ doesn't have to be a quadratic residue. Therefore the condition $(\frac{d}{p})=-1$ of theorem is met.
\end{proof}

\begin{lem}\label{IsomwithField}
 The matrix algebra $Alg [A]=\left\langle E,\,\,A \right\rangle \simeq {{F}_{{{p}^{2}}}}$.
\end{lem}

\begin{proof}
We show that algebra $Alg \left[ M \right]=\left\langle E,A \right\rangle $ is isomorphic to finite field ${{F}_{{{p}^{2}}}}$.
As well-known from Galois theory, a quadratic extension of $\F_p$
can be constructed by involving of any external element.
As well-known from Galois theory, a quadratic extension of ${{F}_{p}}$ can be constructed by involving of any external element $g\in {{F}_{{{p}^{2}}}}\backslash {{F}_{p}}$  relatively to ${{F}_{p}}$. We denote this element by $i$, in particular, for $p=4m+3$ it may be an element satisfying the relation ${{i}^{2}}=-1$.
Note that the matrix of the rotation by 90 degrees, that is a matrix
$$I:= \left( \begin{array}{cc}
  0\,\,\,\,&1 \\
 -1\,\,&0 \\
 \end{array} \right) ={{\rho }_{90}}$$
 satisfies this relation and can used as an example of matrix $A$. In case when $p=4m+3$  such matrix $J:\,\,\varphi (J)=j$,  ${{j}^{2}}=-1$ exists too.

Obviously $\det A=1$, that's why $A\in S{{L}_{2}}({{F}_{p}})$ and ${{\mu }_{A}}(x)$ is irreducible.
We define mapping $\varphi :\,\,{{y}_{1}}A+{{x}_{1}}E\,\to ae+b\lambda ;\,\,\,\,a,b\in {{\text{F}}_{p}}$.
The mapping $\varphi $ can be more broadly described, in $S{{L}_{2}}[{{F}_{p}}]$ such a way that a matrix $A$   is found such that ${{A}^{2}}=-E$, then its e.g. $\lambda $ is assigned to it in the field ${{F}_{{{p}^{2}}}}$, while $\lambda \in {{F}_{{{p}^{2}}}}\backslash {{F}_{p}}$.
$\varphi :\,\,{{y}_{1}}A+{{x}_{1}}E\,\to ae+b\lambda ;\,\,\,\,a,b\in {{\text{F}}_{p}}$.
According to assumption of Lemma the matrix $A$ is semisimple and has not multiple eigenvalues (e.g.) which are not squares in ${{\text{F}}_{p}}$, so ${{\chi }_{A}}\left( x \right)$  is irreducible because of definition of semisimple matrix and condition ${{\lambda }_{1}}\ne {{\lambda }_{2}}$. According to Lemma about Frobenius automorphism its eigenvalues are conjugated in ${{\text{F}}_{{{p}^{2}}}}$.
The method of constructing of $\sqrt{A}$ is the following. Having isomorphism $A\lg \left[ A \right]=\left\langle E,A \right\rangle \simeq {{\text{F}}_{{{p}^{2}}}}$ we set a correspondence $\lambda \leftrightarrow A$ and correspondence between groups operations in ${{\text{F}}_{{{p}^{2}}}}$ and $A\lg \left[ A \right]$.  Therefore solving equation ${{\left( x+\lambda y \right)}^{2}}=\lambda $ relatively coefficients $x,\,\,y\in {{\text{F}}_{p}}$ we obtain coefficients for expression for $\sqrt{A}$ i.e. $\sqrt{A}=x+Ay$.
To prove the isomorphism, we establish a bijection between the generators of the algebra $A\lg \left[ A \right]=\left\langle E,A \right\rangle $ and the field ${{F}_{{{p}^{2}}}}$. It is necessary to establish in more detail that $A\leftrightarrow \lambda $ and  $E\leftrightarrow e$ also  the correspondence between the neutral elements of both structures, i.e. $\varphi \left( {\bar{0}} \right)=0$ where $0$ is the zero matrix.
To complete proof, it remains to show that the kernel of this homomorphism $\varphi $ is trivial. To do this, we show that among the elements of the algebra there are no identical ones. The surjectivity of $\varphi $ is obvious.  From the opposite,  we assume ${{y}_{1}}A+{{x}_{1}}E={{y}_{2}}A+{{x}_{2}}E$, ${{x}_{i}},\,{{y}_{i}}\in {{F}_{p}}$.  Then ${{y}_{1}}A+{{x}_{1}}E={{y}_{2}}A+{{x}_{2}}E$ it yields that
$\left( {{y}_{1}}-{{y}_{2}} \right)E=\left( {{x}_{1}}-{{x}_{2}} \right)A$,
which is impossible since the characteristic polynomial of the matrix $A$ is irreducible but  the characteristic polynomial of the identity matrix is reducible. Therefore, our algebra $A\lg \left[ A \right]$ is isomorphic to the completely linear space of linear polynomials from $E$ and $A$.
In the similar way we prove that polynomial of form $xe+y\lambda $ where $x,\,\,y\in {{F}_{p}}$ do not repeat. The proof is based on oppositive assumption about coinciding ${{x}_{1}}e+{{y}_{1}}\lambda ={{x}_{2}}e+{{y}_{2}}\lambda$ of polynomial with different coefficients. Then equality ${{x}_{1}}e+{{y}_{1}}\lambda ={{x}_{2}}e+{{y}_{2}}\lambda$  implies that $\left( {{y}_{1}}-{{y}_{2}} \right)\lambda =\left( {{x}_{1}}-{{x}_{2}} \right)e$ i.e. ${{y}_{1}}={{y}_{2}}$ and ${{x}_{1}}={{x}_{2}}$ that contradicts to assumption.
\end{proof}

\begin{thm}\label{inFp^2}
If a matrix $A\in G{{L}_{2}}({{F}_{p}})$ is semisimple with different eigenvalues and at least one an eigenvalue ${{\lambda }_{i}} \in F_{p^2} \setminus F_{p}$, $ i\in \{1,2\},  \, p>2$, then $\sqrt{A}\in G{{L}_{2}}({{F}_{p}})$ iff of $A$ satisfies:


 \begin{equation*}
(\frac{{{\lambda }_{i}}}{p})=1
 \, \, in \, \, the \, \, square \, \, extention \, \, that \, \, is \, \, {{F}_{p^2}}.
 \end{equation*}
\end{thm}

\begin{proof}
Firstly, we consider most complex and interesting case when $A$ is not diagonalizable, then ${{\chi }_{A}}\left( x \right)$ is irreducible over  ${{F}_{p}}$. By assumption, the matrix is semisimple and its characteristic polynomial is irreducible. So root $\lambda $ of ${{\chi }_{A}}(x)$ belongs to the quadratic extension of the field ${{F}_{p}}$. Since each element of ${{F}_{{{p}^{2}}}}$ can be presented in form $a+b\lambda ,\,\,\,a,b\in {{F}_{p}}$, then we can construct mapping of matrix algebra generators $E$ and $A$  in generators of ${{F}_{{{p}^{2}}}}$ and apply the aforementioned Lemma \ref{IsomwithField} about isomorphism establish correspondence between property be square in ${{\text{F}}_{{{p}^{2}}}}$ and in  $Alg \left[ A \right]=\left\langle E,A \right\rangle $. If one e.v. ${{\lambda}_{i}}$ is square in $F_{p^2}$ then so is second e.v. because of they are conjugated as roots of characteristic polynomial $\chi_A(x)$ by theorem about Frobenius automorphism (Frobenius endomorphism in perfect field became to be automorphism).
\end{proof}

\begin{exm} 
Consider the matrix $A=-E$, where $E$ is identity matrix in $S{{L}_{2}}({{F}_{3}})$  satisfying conditions of Theorem \ref{inFp^2} because of $(\frac{-1}{9})=1$ in ${{\text{F}}_{9}}$.
And there exists the matrix $\left( \begin{array}{cc}
   0\,\,\,&2 \\
  -2\,\,&0
\end{array} \right)\in S{{L}_{2}}({{F}_{3}})$  is square root for $A$. Indeed ${{I}^{2}}=-E$.

Another root of this equation ${{X}^{2}}=A$, where $A$ is matrix of elliptic type realizing rotation on 90 degrees ${{\rho }_{90}}=\left( \begin{array}{cc}
  \,\,\, 0\,\,\,\,1 \\
  -1\,\,0 \\
\end{array} \right)=I$ because of ${{I}^{2}}=-E$, is matrix of parabolic type.

The matrix $2I$ is the square in $G{{L}_{2}} \left ({{F}_{3}} \right)$ because of existing such an element ${{\left( \begin{array}{cc}
   \,\,\,1\,\,\, 1 \\
  -1\,\,\,\,1 \\
\end{array} \right)^{2}}}=2\left(\begin{array}{cc}
   \,\,  \ 0\,\,\,\,1 \\
  -1\,\,0 \\
\end{array} \right)=2I.$

\end{exm}


\begin{exm}
Consider the diagonal matrix $A \in S{{L}_{2}}({{F}_{3}})$ emphasizing the need for the condition $F_{p^2} \setminus F_{p}$ in Theorem \ref{inFp^2} for semisimple matrix.
It is easy to verify the absence of root from $A = \left( \begin{array}{cc}
   1\, \,\,\,\,0 \\
 \, 0\, -1 \\
\end{array} \right)$ in $S{{L}_{2}}({{F}_{3}})$.
\end{exm}

\begin{exm}
Consider the diagonal matrix $A \in S{{L}_{2}}({{F}_{3}})$ emphasizing the need for the condition $F_{p^2} \setminus F_{p}$ in Theorem \ref{inFp^2} for semisimple matrix.
It is easy to verify the absence of root from $A = \left( \begin{array}{cc}
   1\, \,\,\,\,0 \\
 \, 0\, -1 \\
\end{array} \right)$ in $S{{L}_{2}}({{F}_{3}})$.
\end{exm}

\begin{thm} \label{diagonal form}  If a matrix $A\in S{{L}_{2}}({{F}_{p}})$ ($A \in GL({F}_{p})$) possesses diagonal Jordan form over ${{\text{F}}_{p}}$, then $\sqrt{A}\in S{{L}_{2}}({{F}_{p}})$ ($GL({F}_{p})$) if and only  if $(\frac{{{\lambda }_{1}}}{p})=1$ and $(\frac{{{\lambda }_{2}}}{p})=1$ over ${{\text{F}}_{p}}$.
\end{thm}

\begin{proof}
From condition $(\frac{{{\lambda }_{1}}}{p})=1$ and $(\frac{{{\lambda }_{2}}}{p})=1$ it is followed, that ${{\mu }_{A}}\left( x \right)$ is reduced over ${{\text{F}}_{p}}$. Therefore why ${{\mu }_{1}},\,\,{{\mu }_{2}}\in {{\text{F}}_{p}}$ exist ${{\mu }_{B}}\left( x \right)$ over ${{\text{F}}_{p}}$ exists for matrix $B:\,\,{{B}^{2}}=A$. Assume that  $(\frac{{{\lambda }_{1}}}{p})=-1,\,\,(\frac{{{\lambda }_{2}}}{p})=-1$ prove, that while $\sqrt{A}\notin S{{L}_{2}}({{F}_{p}})$. We use proof by contradiction. Let  $(\frac{{{\lambda }_{1}}}{p})=-1,\,\,(\frac{{{\lambda }_{2}}}{p})=-1$  therefore roots from eigenvalues ${{\lambda }_{1}},\,\,{{\lambda }_{2}}$ in general belongs to ${{\text{F}}_{{{p}^{2}}}}$ while its roots ${{\mu }_{1}},\,\,{{\mu }_{2}}$ is not conjugated as roots from different values of ${{\lambda }_{1}},\,\,{{\lambda }_{2}}$.

Let's find minimal polynomial for $B=\sqrt{A}$.
    Minimal polynomial of matrix $B$ is ${{\mu }_{B}}(x)={{x}^{2}}-bx+c$ and it has different roots ${{\mu }_{1}},\,\,{{\mu }_{2}}$, where ${{\mu }_{1}}+{{\mu }_{2}}=Tr(B)=b$ e $\det B={{\mu }_{1}}{{\mu }_{2}}$. From the existence of diagonal representation for $A$ reducibility of ${{\mu }_{A}}\left( x \right)$ follows. From the reducibility of ${{\mu }_{A}}\left( x \right)$ over ${{\text{F}}_{p}}$ and the fact that ${{\lambda }_{1}}\ne {{\lambda }_{2}}$ follows ${{\mu }_{1}},\,\,{{\mu }_{2}}$ is not  conjugated as the roots of different values of ${{\lambda }_{1}},\,\,{{\lambda }_{2}}$ and it is obvious that $\mu _{1}^{2}\,\ne \,\mu _{2}^{2}$. But the root ${{\mu }_{1}}$ is conjugated with $-{{\mu }_{1}}$ and ${{\mu }_{1}}\in {{\text{F}}_{{{p}^{2}}}}\backslash {{\text{F}}_{p}}$. But $-{{\mu }_{1}}$ is also a root, since ${{\left( \pm {{\mu }_{1}} \right)}^{2}}={{\lambda }_{1}}$ therefore it can be the root for ${{\mu }_{B}}(x)$. Similar situation is with root ${{\mu }_{2}}$ and $-{{\mu }_{2}}$. Therefore, we indicated as many as 4 roots for ${{\mu }_{B}}(x)$ but $B\in S{{L}_{2}}({{F}_{{{p}^{2}}}})$ therefore $\deg \left( {{\mu }_{B}}(x) \right)=2$. This contradiction arises from the assumption that $\sqrt{A}\in S{{L}_{2}}({{F}_{p}})$ on condition $(\frac{{{\lambda }_{1}}}{p})=-1,\,\,(\frac{{{\lambda }_{2}}}{p})=-1$.
\end{proof}

\subsection{Matrix roots of higher powers}

\textbf{Hypothesis}.  If we consider vector space over the same perfect field $k$ over which we consider $G{{L}_{2}}(k)$ then we have $\sqrt[3]{A}\in Span\{A,\,\,E\}$ over  $k$, where Span is linear span.

For proof we take into account Cayley-Hamilton's equation and apply a reduction transformation of the second degree, then we get an expression similar to
$\sqrt[3]{A}=\frac{A+tr\left( \sqrt[3]{A} \right)\det (\sqrt[3]{A})}{t{{r}^{2}}\left( \sqrt[3]{A} \right)- \sqrt[3]{\det \left( A \right)}}$
which after transformations yields expression of root $\sqrt[3]{A}=\lambda A+\beta $, where  $\lambda, \beta \in k$.

If we restrict the set of matrices to the group $GL_2(F_p)$, then the formulation of the theorem will take the next form.

\textbf{Proposition}.  If $B\in G{{L}_{2}}({{\mathbb{F}}_{p}})$ is root of equation ${{X}^{3}}=A$, then

$$B=\frac{A+tr(\sqrt[3]{A})\sqrt[3]{\det (A)}} {\left(tr{ \sqrt[3]{A} }\right)^{2}-\sqrt[3]{\det (A)}},$$
 where $A\in G{{L}_{2}}({{\mathbb{F}}_{p}})$.

\begin{proof}
Proof.  If $\sqrt[3]{A}\in G{{L}_{2}}({{\mathbb{F}}_{p}})$ then we consider Cayley-Hamilton equation (C.H.E.)${{A}^{3}}-tr\left( A \right){{A}^{2}}+\left( {{\lambda }_{1}}{{\lambda }_{2}}+{{\lambda }_{1}}{{\lambda }_{3}}+{{\lambda }_{2}}{{\lambda }_{3}} \right)A-\det \left( A \right)=0$.
Note, that $tr{{\left( A \right)}^{2}}={{\left( {{\lambda }_{1}}+{{\lambda }_{2}}+{{\lambda }_{3}} \right)}^{2}}=\lambda _{1}^{2}+\lambda _{2}^{2}+\lambda _{3}^{2}-\left( {{\lambda }_{1}}{{\lambda }_{2}}+{{\lambda }_{1}}{{\lambda }_{3}}+{{\lambda }_{2}}{{\lambda }_{3}} \right)$.

Consider  C.H.E.  for  $A:\, dimA=2$,  ${{A}^{2}}-tr\left( A \right)\cdot A+\det \left( A \right)\cdot 1=0$.
Multiplying last equation on $A$ admit us obtain the chain of transformation:
\begin{align} \label{KHE}
   {{A}^{3}} & =\left( tr\left( A \right)A-\det \left( A \right) \right)A=tr\left( A \right){{A}^{2}}- \det \left( A \right)A= \nonumber\\
   & =tr\left( A \right)\left( tr\left( A \right)A-\det \left( A \right) \right)-\det \left( A \right)A=  \nonumber\\[-3mm]
   & ~ \\[-3mm]
  & =tr{{\left( A \right)}^{2}}A-tr\left( A \right)\det \left( A \right)-\det \left( A \right)A= \nonumber\\
 & =\left( tr{{\left( A \right)}^{2}}-\det \left( A \right) \right)A-tr\left( A \right)\det \left( A \right). \nonumber
\end{align}

By applying substitute matrix $\sqrt[3]{A}$ instead of $A$ we express

\begin{equation}\label{2}
\sqrt[3]{A}=\frac{A + tr\left(     \sqrt[3]{A} \right)\sqrt[3]{\det A}}{tr^{2}{{\left( \sqrt[3]{A} \right)}}-\sqrt[3]{\det \left( A \right)}}.					
\end{equation}

Thus, $\sqrt[3]{A}=\frac{A\,+tr\left( \sqrt[3]{A} \right)\sqrt[3]{\det \left( A \right)}}{\,\left( t{{r}^{2}}\left( \sqrt[3]{A} \right)-\sqrt[3]{\det \left( A \right)} \right)}$.

Note that $\det \left( \sqrt[3]{A} \right)=\sqrt[3]{\det \left( A \right)}$ because of determinant is homomorphism.

But $tr\left( \sqrt[3]{A} \right)$ is still not computed.
 From \eqref{KHE}
 we conclude \\ ${A}^{3}=\left( tr{{\left( A \right)}^{2}}-\det \left( A \right) \right)A-tr\left( A \right)\det \left( A \right).$ Computing a trace from both sides we obtain
$tr\left( {{A}^{3}} \right)=tr{{\left( A \right)}^{3}}-3\det \left( A \right)tr\left( A \right)$.

We put $\sqrt[3]{A}$ instead of $A$, then we get
$tr\left( A \right)=tr{{\left( \sqrt[3]{A} \right)}^{3}}-3\sqrt[3]{\det A}tr\left( \sqrt[3]{A} \right).$

We need to solve
$tr\left( A \right)=tr{{\left( \sqrt[3]{A} \right)}^{3}}-3\sqrt[3]{\det A}tr\left( \sqrt[3]{A} \right)$.

We denote $\sqrt[3]{A}$ by $X$ and obtain the equation

$${{X}^{3}}-3\sqrt[3]{\det \left( A \right)}X-tr\left( A \right)=0.$$

The \emph{solvability} of this equation over base field $\mathbb{F}_p$ is equivalent to the \emph{existence} of a trace $\sqrt[3]{A}$ in the base field.

In view of this we derive number of roots in $S{{L}_{2}}\left( \text{F} \right)$.
Let $p=3\sqrt[3]{\det A}$, $q=\tr A$  then we have 1 root in this field if \[D=\frac{{{p}^{3}}}{3}+\frac{{{q}^{2}}}{2}=-\frac{27\det A}{3}+\frac{{{(trA)}^{2}}}{2}>0.\]
And we have 3 different roots if $-\frac{27\det A}{3}+\frac{{{(trA)}^{2}}}{2}<0,$ in case $D=0$ then there are one root and 2 multiple roots over this field.

Now we consider singular case:
\begin{itemize}
\item $(tr{B})^2 - \det{B}=0$, where $B= \sqrt[3]{A}$. \\
In this case in view of $\det{B} = (tr{B})^2$ and from \eqref{2} we obtain $$A=B^3 =  - tr{B} \det{B} \cdot E = -(tr{B})^3 \cdot E.$$
From that we can compute $tr{B}$ as a root of the equation $x^3+\dfrac{tr{A}}{2}=0$.

\item If $B^3=0$, then it's minimal canceling polynomial is $X^2$ or $X$. By Celly Hamilton equation (C.H.E) $B^2 - tr{B} \cdot B + \det{B} \cdot E = 0$, which leads us to $tr{B}=0, \det{B}=0$.
\end{itemize}
\end{proof}

For generalization on a matrix ring we reformulate previous statement in the following way.

\begin{prop} Let $A\in M_2(\mathbb{F}_p)$. Then it's cube roots $R = \{B \in M_2(\mathbb{F}_p)  \mid B^3 = A\}$ can be obtained as follows:
\begin{enumerate}
\item If $A=0$, then $R=\{B \in M_2(\mathbb{F}_p) \mid \det{B}=0,~\tr{B}=0 \}$;
\item If $A = c^3 E$, where $c \in \mathbb{F}_p / 0$, then $R=\{c \cdot B \in M_2(\mathbb{F}_p) \mid B^3 = E\}$;
\item In other cases $R \subset \left\{B \in M_2(\mathbb{F}_p) \left|  B=\dfrac{A + ab \cdot E}{a^2-b} \right., a=\tr{ \sqrt[3]{A}}, \right.$\\ $\left.b^3=\det{A}, a^3-3ab = \tr{A} \,  \right\}$.
\end{enumerate}

\end{prop}

\begin{proof}
\begin{enumerate}
\item If $B^3=0$, then it's minimal canceling polynomial is $X^2$ or $X$. By Celly Hamilton equation (C.H.E) $B^2 - \tr{B} \cdot B + \det{B} \cdot E= 0$, which leads us to $\tr{B}=0,   \,\,  \det{B}=0$;
\item If $B$ is a solution of $X^3 - c^3 \cdot E=0$, then it's easy to see that $B' = c^{-1} B$ is a solution of $X^3 - E=0$;
\item Consider C.H.E for $B$:
$$B^2 - \tr{B} \cdot B + \det{B} \cdot E = 0.$$
Multiplying last equation by $B$ we proceed with the following chain of transformations:
\begin{multline*}
 B^3=
(\tr{B}\cdot B - \det{B} \cdot E) \cdot B =
\tr{B} \cdot B^2 - \det{B}\cdot B =
\tr{B} (\tr{B}\cdot B - \det{B} \cdot E) - \det{B} \cdot B=\\
=(\tr{B})^2 \cdot B - \tr{B} \det{B} \cdot E- \det{B} \cdot B =
((\tr{B})^2 - \det{B}) \cdot B - \tr{B} \det{B} \cdot E.
\end{multline*}
If $(\tr{B})^2 - \det{B}=0$, then we obtain $A=B^3 =  -\tr{B} \det{B} \cdot E =(-\tr{B})^3 \cdot E$, which leads us to previous cases.\\
\\
Otherwise $(\tr{B})^2 - \det{B}\not=0$ and we express $B$:
$$ B = \dfrac{B^3+\tr{B}\det{B}~E}{(\tr{B})^2-\det(B)}.$$

Now since $B^3=A$ we conclude $\det{A} = \det{B^3} = (\det{B})^3$ and hence $\det{B}$ is a root of polynomial $x^3 - \det{A} = 0$.

Last thing one remains to find $\tr{B}$.
By computing trace from both sides of $A=((\tr{B})^2 - \det{B}) \cdot B - \tr{B} \det{B} \cdot E$ we get:
$$\tr{A} = (\tr{B})^3 -3 \tr{B} \det{B}$$
From which we conclude that $\tr{B}$ is a root of $x^3-3 \det{B} \cdot x-\tr{A}=0$.
\end{enumerate}
\end{proof}

In general case we define complete symmetric polynomial of $n$-th degree in two variables: $$h_n(x,y) = \sum\limits_{k=0}^n x^ky^{n-k}.$$
In view of the fundamental theorem of symmetric polynomials there is one unique polynomial $Q(x,y) \in \mathbb{F}_p[x,y]$, such that:
$Q(e_1, e_2) = h_n$, where $e_1 = x+y$, $e_2 = xy$ --- elementary symmetric polynomials.

Likewise we determine the power symmetric polynomial of $n$-th degree in two variables: $$p_n(x,y) = x^n + y^n.$$
And polynomial  $P(x,y) \in \mathbb{F}_p[x,y]$, such that $P(e_1, e_2) = p_n$.

Now we prove the following lemma.
Let us define sequences $s_n = \tr{B} ~s_{n-1} + t_{n-1}$ and
 $t_n = - \det{B}~ s_{n-1}$ with initial conditions $s_1 = 1, t_1 = 0$, $s_2 = tr B$ and $t_2 = - det B$.
The parameters $\tr(B)$ and the determinant of matrix $B$ can be calculated thanks to Lemma \ref{eigenval} or by using the inversion of the Chebyshev polynomial.

\begin{lem}
Sequences $s_n$, $t_n$ satisfy recurrent equation with characteristic polynomial $c(x)$ which is also characteristic polynomial for matrix $B$.

\end{lem}

\begin{proof}

$X^n = X \cdot X^{n-1} \underset{c(X)}  \equiv X \cdot (s_{n-1}X + t_{n-1} E) = s_{n-1}X^2 + t_{n-1} X \underset{c(X)}  \equiv s_{n-1}(\tr{B} X - \det{B} \cdot E) + t_{n-1} X = (s_{n-1} \tr{B} + t_{n-1}) X - s_{n-1} \det{B} \cdot E$

Or by definition of $s_n$ and $t_n$:
\begin{equation}\label{recur} \begin{cases} s_n = \tr{B} ~s_{n-1} + t_{n-1}\\
t_n = - \det{B}~ s_{n-1}
\end{cases}\end{equation}

By summing up first expression from \eqref{recur} multiplied by $\det{B}$ with the second one multiplied by $\tr{B}$ we get:
$$ \det{B}~ s_n + \tr{B}~ t_n = \det{B}~ t_{n-1} $$
or
$$ \det{B}~ s_n =  \det{B}~ t_{n-1} - \tr{B}~ t_n $$

Substituting into second equation of (2) we obtain:
$$t_n-\tr{B}~t_{n-1}+\det{B}~t_{n-2}=0$$
Since $s_n $ and $t_n$ are linearly dependant it follows that $s_n$ satisfy the same recurrent.
\end{proof}
\begin{thm}
 Let $n \geqslant 3$ and $A\in M_2(\mathbb{F}_p)$, $a=\tr {\sqrt[3]{A}}$.  If $A \not = c \cdot E$ for any $c \in \mathbb{F}_p$ and $R = \{B \in M_2(\mathbb{F}_p)  \mid B^n = A\}$  set of it's  $n$-th roots, then next inclusion follows:
$$R \subset \left\{B \in M_2(\mathbb{F}_p) \left|  B=\dfrac{A+b~ Q_{n-2}(a,b) \cdot I}{Q_{n-1}(a,b)} \right.,~ b^n =\det{A}, ~P_n(a,b) = \tr{A}.  \right\}$$


\end{thm}

\begin{proof}

Let $B \in M_2(\mathbb{F}_p)$ be a root of equation $X^n = A$. Also consider it's C.H.E. $$c(X) = X^2 - \tr{B} X + \det{B} \cdot E.$$
Then $X^n \underset{c(X)}  \equiv s_n X + t_n E$ for some $s_n, t_n \in \mathbb{F}_p$ and since $c(B) = 0$ we have
\begin{equation}\label{recursion for A}
A = s_n B +t_n E.
 \end{equation}


Since $X^1 \underset{c(X)}  \equiv X + 0 \cdot E$ and $X^2 \underset{c(X)}  \equiv \tr{B} X - \det{B} \cdot E$, we have $s_1 = 1, t_1 =0, s_2 = \tr{B}$ and $t_2 = - \det{B}$.

Consider algebraic closure of $\mathbb{F}_p$ --- $\widehat{\mathbb{F}_p}$. Let $\lambda_1, \lambda_2$ be roots of $c(x)$ in $\widehat{\mathbb{F}_p}$ (eigenvalues of B).

\begin{enumerate}

\item If $\lambda_1 \not=\lambda_2$ and $\lambda_1 \lambda_2 = \det{B} \not= 0$:

$$s_n = c_1 \lambda_1^n + c_2 \lambda_2^n, ~t_n = c_1' \lambda_1^n + c_2' \lambda_2^n$$

In cases $n=1,2$ for $s_n$ we get:

$$\begin{cases} c_1 \lambda_1 + c_2 \lambda_2 = 1,\\
c_1 \lambda_1^2 + c_2 \lambda_2^2 = \tr{B}
\end{cases}$$

Solving the system using Kramer's rule we obtain: $$c_1 = \dfrac{\lambda_2^2 - \lambda_2  \tr{B}}{\lambda_1 \lambda_2^2 - \lambda_1^2 \lambda_2} = -\dfrac{1 }{\lambda_2 - \lambda_1}, ~c_2 =  \dfrac{\lambda_1 \tr{B} - \lambda_2^2}{\lambda_1 \lambda_2^2 - \lambda_1^2 \lambda_2} =  \dfrac{1}{\lambda_2 - \lambda_1}$$

Substituting constants

\begin{equation} s_n = \dfrac{\lambda_2^n - \lambda_1^n}{\lambda_2 - \lambda_1} = h_{n-1}(\lambda_1, \lambda_2) \end{equation}

In cases $n=1,2$ for $t_n$ we get:

$$\begin{cases} c_1' \lambda_1 + c_2' \lambda_2 = 0,\\
c_1' \lambda_1^2 + c_2' \lambda_2^2 = - \det{B}
\end{cases}$$

Solving the system using Kramer's rule we obtain: $$c_1' = \dfrac{\lambda_2 \det{B}}{\lambda_1 \lambda_2^2 - \lambda_1^2 \lambda_2} = \dfrac{\lambda_2 }{\lambda_2 - \lambda_1}, ~c_2' = - \dfrac{\lambda_1 \det{B}}{\lambda_1 \lambda_2^2 - \lambda_1^2 \lambda_2} =  -\dfrac{\lambda_1 }{\lambda_2 - \lambda_1}$$

Substituting constants

\begin{equation} t_n = \dfrac{\lambda_1^n\lambda_2-\lambda_1\lambda_2^n}{\lambda_2 - \lambda_1}= -\det{B} \cdot \dfrac{\lambda_1^{n-1}-\lambda_2^{n-1}}{\lambda_1-\lambda_2} = -\det{B}~ h_{n-2}(\lambda_1, \lambda_2) \end{equation}

\item In general case for each $n \ge 3$ we consider polynomial $D_n(\lambda_1, \lambda_2) = h_{n-1} - \tr{B} h_{n-2} + \det{B} h_{n-3}$. It's a continuous function of variables $\lambda_1, \lambda_2$.

Previously we proved that $D_n(\lambda_1, \lambda_2)=0$ if $\lambda_1 \not= \lambda_2$ and $\lambda_i \not = 0$.

From continuity follows that $D_n(\lambda_1, \lambda_2)=0$  $\forall \lambda_1, \lambda_2$ and hence formulas (3) and (4) are fulfilled $\forall \lambda_1, \lambda_2$.
\end{enumerate}

Now that we have found $s_n$ and $t_n$ we return to equation \eqref{recursion for A}. If $s_n=0$, then $A = t_n I$ which contradicts conditions of the theorem. Dividing both sides by $s_n$ we get formula

$$B = \dfrac{A - t_n I}{s_n} = \dfrac{A + \det{B}~ h_{n-2}(\lambda_1, \lambda_2) \cdot I}{h_{n-1}(\lambda_1, \lambda_2)}=\dfrac{A + \det{B}~ Q_{n-2}(\tr{B}, \det{B}) \cdot I}{Q_{n-1}(\tr{B}, \det{B})}$$

The last thing remaining is to express $\det{B}$ and $\tr{B}$ in terms of $A$.

Since $\det{A} = \det{B^n} = \det{B}^n$, $\det{B}$ can be obtain as root of polynomial $x^n=\det{A}$.

To find $\tr{B}$ we compute trace from both sides of \eqref{recursion for A}:
\begin{multline*} \tr{A} =  \tr{B}~ s_n +  2~ t_n = \tr{B}~ h_{n-1}(\lambda_1, \lambda_2) - 2 \det{B}~h_{n-2}(\lambda_1, \lambda_2) =\\=  h_{n}(\lambda_1, \lambda_2) -  \lambda_1\lambda_2~h_{n-2}(\lambda_1, \lambda_2) = \lambda_1^n+\lambda_2^n = p_n(\lambda_1, \lambda_2) = P_n(\tr{B}, \det{B}). \end{multline*}
\end{proof}

\section{Conclusion }
New linear group which is storage of square roots from $ SL_2({\mathbb{F}_p})$ is found and investigated by us.

The analytic formula of cubic square from matrix  $ SL_2({\mathbb{F}})$ is founded.
The the analytical formulas of square and 4-th power roots in $ SL_2({\mathbb{F}_p})$, $ ESL_2({\mathbb{F}_p})$, for any prime $p$,  as well as in $ SL_2({Z})$  $ ESL_2({Z})$ and in $ SL_2({k}), ESL_2({k})$, where $k$ is arbitrary perfect field, is found by us.

The analytic formula of cubic square from matrix  $ SL_2({\mathbb{F}})$ is founded.
The analytical formula of square and 4-th power roots in $ SL_2({\mathbb{F}_p})$, $ ESL_2({\mathbb{F}_p})$, for any prime $p$,  as well as in $ SL_2({Z})$,  $ ESL_2({Z})$ and in $ SL_2({k}), ESL_2({k})$, where $k$ is arbitrary perfect field, is found by us.
Furthermore the recursive formula of square and $n$-th power roots in $ SL_2({\mathbb{F}_p})$ is found by us.

The criterions of matrix equation $X^2=A$ solvability over different linear groups with respect to matrix classification by its $tr(A)$ and type of space contracting is found and proved in this paper.

The criterion of roots existing for different classes of matrix --- simple and semisimple matrixes from $ SL_2({\mathbb{F}_p})$, $ SL_2({\mathbb{Z}})$ are established.

If a matrix $A\in G{{L}_{2}}({{F}_{p}})$ is semisimple with different eigenvalues and at least one an eigenvalue ${{\lambda }_{i}} \in F_{p^2} \setminus F_{p} $, $ i\in \{1,2\}$, then $\sqrt{A}\in G{{L}_{2}}({{F}_{p}})$ iff $A$ satisfies:

 \begin{equation*}
(\frac{{{\lambda }_{i}}}{p})=1
 \, \, in \, \, the \, \, algebraic \, \, extention \, \, of \, \, degree \, \, 2 \, \, that \, \, is \, \, {{F}_{p^2}}.
 \end{equation*}

\textbf{Acknowledgement}. Special thanks to Natalia Vladimirovna Maslova for seminars provided by her and her good questions by the topic.


%
%


\end{document}